\documentclass[10pt]{article}

\pdfoutput=1

\usepackage{amsmath}
\usepackage{amsthm}
\usepackage{amsfonts}
\usepackage{amssymb}
\usepackage{mathtools}
\usepackage{mathrsfs}
\usepackage{stmaryrd}
\usepackage{enumitem}
\usepackage{tikz-cd}
\usepackage{caption}
\usepackage[english]{babel}
\usepackage{csquotes}
\usepackage{cite}
\usepackage{hyperref}

\newcommand{\for}{\hspace{0.5pt} :\hspace{0.5pt}} 
\renewcommand{\subset}{\subseteq}

\DeclareMathOperator{\ima}{im}

\newcommand{\clos}[1]{\overline{#1}}
\DeclareMathOperator{\frontier}{fr}

\newcommand{\real}{\mathbb{R}}
\newcommand{\nat}{\mathbb{N}}

\renewcommand{\leq}{\leqslant}
\renewcommand{\geq}{\geqslant}
\newcommand{\abs}[1]{\left\lvert #1 \right\rvert}
\newcommand{\norm}[1]{\left\lVert #1 \right \rVert}
\newcommand{\length}[1]{\left\lvert #1 \right\rvert}
\newcommand{\grad}{\nabla}
\DeclareMathOperator{\rankaux}{rk}
\newcommand{\rank}[2]{\rankaux_{#1}(#2)}
\newcommand{\tangent}[2]{T_{#1}(#2)}

\newcommand{\formser}[2]{#1 \llbracket #2 \rrbracket}
\DeclareMathOperator{\supp}{supp}

\newcommand{\genser}[2]{#1 \llbracket #2^* \rrbracket}

\newcommand{\struct}{\real_H}
\newcommand{\graph}[1]{\Gamma(#1)}
\newcommand{\functionalg}[2]{\mathcal{A}_{#1, #2}}
\newcommand{\germalgnotH}[1]{\mathcal{A}_{#1}}
\newcommand{\germalg}[1]{\mathcal{A}_{#1}(H)}
\newcommand{\polydisk}[2]{I_{#1, #2}}
\DeclareMathOperator{\taylor}{\mathbf{T}}

\theoremstyle{definition}
\newtheorem{defn}{Definition}[section]
\newtheorem{remark}[defn]{Remark}
\theoremstyle{plain}
\newtheorem{thm}[defn]{Theorem}
\newtheorem{prp}[defn]{Proposition}
\newtheorem{cor}[defn]{Corollary}
\newtheorem{lem}[defn]{Lemma}

\author{R\'emi Gu\'enet}
\title{A quasianalytic class with weakly smooth germs}

\begin{document}

\maketitle
\tableofcontents
\pagebreak

\section{Introduction}

Recall that an expansion \(\langle M; <, \dots \rangle\) of a dense
linear order is called
\textbf{o-minimal} when the definable subsets of \(M\) are exactly the finite
unions of points and intervals. By the Tarski-Seidenberg Theorem, the
structure \(\real_\text{alg} = \langle \real; <, +, -, \cdot, 0, 1
\rangle\) has quantifier elimination. It is easy to deduce that it is
also model-complete and o-minimal.

Given a power series \(f \in \formser \real {X_1, \dots, X_n}\) that
converges in a neighborhood of \(I^n\), where \(I = [-1, 1]\), we
define
\begin{equation*}
  \widetilde f(x) =
  \begin{cases}
    f(x), &x \in I^n \\
    0, &x \not \in I^n
  \end{cases}
\end{equation*}
and call the function \(\widetilde f\) a \textbf{restricted analytic
function}. Consider \(\real_\text{an}\), the structure
\(\real_\text{alg}\) augmented by a function symbol for each
restricted analytic function. In
\cite{vddries-macintyre-marker-restricted-analytic-exponentiation},
the authors give a
quantifier elimination result for the structure \(\langle
\real_\text{an}; \exp, \log \rangle\) and they deduce that the
structure \(\real_{\text{an}, \exp} = \langle \real_\text{an};
\exp\rangle\) is o-minimal.

In particular, the structure \(\real_\text{an}\) is
o-minimal. Actually, this is already observed in
\cite{vddries-generalization-tarski-seidenberg}, where this
result is obtained as a consequence of Gabrielov's Theorem of the
complement proven in \cite{gabrielov-complement}. This method, which
does not rely on
quantifier elimination, is adapted in
\cite{vdDries-Speissegger-generalized-power-series} in the following
way. A \textbf{generalized series} \(f \in \genser \real X\), where \(X = (X_1,
\dots, X_m)\), is a formal sum
\[f = \sum_{\alpha} c_\alpha X^\alpha,\]
where \(\alpha\) runs over the set \([0, \infty)^m\) and such that
there are well ordered sets \(S_1, \dots, S_m \subset [0, \infty)\) with
\(\{\alpha \for c_\alpha \neq 0\} \subset S_1 \times\dots\times
S_m\). The set \(\supp (f) \coloneq \{\alpha \for c_\alpha \neq 0\}\) is
called the \textbf{support} of \(f\). The
condition on the support guarantees that these series can be added and
multiplied in the usual way. If \(f \in \genser \real X\), we say that
it converges when the series
\[\sum_{\alpha \in \supp (f)} c_\alpha  x^\alpha\]
converges absolutely on \(I^m\). In this case, we define
\begin{equation*}
  \widetilde f(x) =
  \begin{cases}
    f(x), &x \in I^m \\
    0, &x \not \in I^m
  \end{cases}
\end{equation*}
as in the analytic case. By adapting Gabrielov's Theorem of the
complement and by giving a suitable monomialization procedure, the
authors of \cite{vdDries-Speissegger-generalized-power-series} show
that the expansion of \(\real_\text{alg}\) by function symbols for
each \(\widetilde f\) is o-minimal. \\

Later on, in \cite{rolin-speissegger-wilkie-denjoy-carleman-classes},
the authors generalize this method to series that do not necessarily
converge. For each \(r = (r_1, \dots, r_n) \in (0, \infty)^n\), define
the \textbf{polydisk} \(\polydisk n r\) as
\[\polydisk n r = (-r_1, r_1) \times \dots \times (-r_n, r_n).\]
Then, for each \(n \geq 0\) and each \(r \in (0, \infty)^n\) consider
an algebra \(\functionalg n r\) of functions \(f \colon \polydisk n r
\to \real\) and let \(\germalgnotH n\) be the algebra of germs at
\(0\) of functions in \(\functionalg n r\). We assume that each germ
in \(\germalgnotH n\) is a \(C^\infty\) germ, meaning that it has a
\(C^\infty\) representative. Then, there is a morphism \(\taylor\colon
\germalgnotH n \to \formser \real X\) that takes a germ \(f \in
\germalgnotH n\) and that returns its Taylor series at \(0\). We say
that the algebras \(\germalgnotH n\) are \textbf{quasianalytic} when
this morphism is injective. In practice, if \(f \in \germalgnotH n\),
we can view \(f\) as a sum of the series \(\taylor(f)\), even when
this series does not converge in any neighborhood of \(0\). Under this
assumption, and other technical hypotheses about the algebras
\(\functionalg n r\), it is then shown in
\cite{rolin-speissegger-wilkie-denjoy-carleman-classes} that the
structure \(\langle \real_\text{alg}; (\widetilde f)_{f \in
  \mathcal{A}}\rangle\) is o-minimal and model-complete, where
\begin{equation*}
  \widetilde f(x) =
  \begin{cases}
    f(x), &x \in \polydisk n r \\
    0, &x \not \in \polydisk n r
  \end{cases}
\end{equation*}
for \(f \in \functionalg n r\).

This same method can also be applied with generalized series, although
more care is necessary in this latter case. An example of this method
is given in \cite{kaiser-rolin-speissegger-non-resonant} while Rolin
and Servi provide a proof of o-minimality and model-completeness under
general assumptions in \cite{rolin-servi-monomialization}. The authors
of the latter paper also remark that most known o-minimal structures
satisfy the hypotheses of their theorem. 
Furthermore, in \cite[Section
4]{rolin-servi-monomialization}, the authors also use o-minimality in the
form of the o-minimal Preparation Theorem, which is proven in 
\cite{vddries-speissegger-preparation}, to derive
a quantifier elimination result. However,
due to the hypotheses made, all the structures to which these results
apply necessarily have \(C^\infty\) cell-decomposition. \\

In \cite{legal-rolin-not-c-infty}, Le Gal and Rolin
show that there
exists an o-minimal expansion of the real field which does not admit
\(C^\infty\) cell-decomposition. Recall that a germ \(f\) of
function \(\real^n \to \real\) at \(0\) is called \textbf{weakly
  \(C^\infty\)} if it is
\(C^p\) for all \(p \geq 0\). This means that it has \(C^p\)
representatives for every \(p \geq 0\) but these can be defined on
smaller and smaller neighborhoods as \(p \to \infty\). Thus, in
particular, this does not imply that \(f\) has a \(C^\infty\)
representative. However, it still makes sense to talk about the Taylor
series of \(f\) since \(f\) has partial derivatives of every
order. Therefore, the concept of quasianalyticity still makes sense in
this context and the authors use the same ideas as the ones in
\cite{rolin-servi-monomialization} in order to prove o-minimality. The
main difference is that they cannot prove a global version of the
Fiber Cutting Lemma so that they have to use a local version of the
result instead. 

Firstly, let us recall some notations from
\cite{legal-rolin-not-c-infty}. We let \(\mathcal{W}_n\) be the
algebra of weakly \(C^\infty\) germs at \(0\) of functions \(\real^n
\to \real\). 

\begin{defn}
  \label{defn:smallest-algebras}
  Let \(H \colon \real \to \real\) be a function whose germ at \(0\)
  is weakly \(C^\infty\). We define \(\mathcal{A}(H) =
  (\mathcal{A}_n(H))\) to be the smallest family of subalgebras
  \(\mathcal{A}_n(H) \subset \mathcal{W}_n\) such that
  \begin{itemize}
  \item The germ of \(H\) belongs to \(\mathcal{A}_1(H)\) and polynomial
    germs in \(n\) variables are in \(\mathcal{A}_n(H)\).
  \item If the germ \(f \in \germalg n\) vanishes on the hyperplane
    \(x_i = 0\) then the germ which continuously extends
    \(\frac{f}{x_i}\) is in \(\germalg n\). 
  \item If \(f \in \germalg m\) and \(g_1, \dots, g_m \in
    \germalg n\) are such that \(g_i(0) = 0\) for every \(1 \leq i \leq
    m\), then \(f(g_1, \dots, g_m) \in  \germalg n\). 
  \item Let \(f \in \germalg n\) be such that \(f(0) = 0\) and
    \(\frac{\partial f}{\partial x_n}(0) \neq 0\). Then, the germ \(\varphi
    \in \mathcal{W}_{n-1}\) such that \(f(x', \varphi(x')) = 0\) is in
    \(\germalg{n-1}\). 
  \end{itemize}
\end{defn}

In \cite[Section 2]{legal-rolin-not-c-infty}, Le Gal and Rolin prove
the following theorem.

\begin{thm}
  \label{thm:properties-H}
  There exists a function \(H \colon \real \to \real\) which satisfies
  the following conditions.
  \begin{itemize}
  \item The germ of \(H\) at \(0\) is weakly \(C^\infty\) but not
    \(C^\infty\).
  \item The restriction of \(H\) to the complement of any neighborhood
    of \(0\) is piecewise given by finitely many polynomials.
  \item The algebras \(\germalg n\) are quasianalytic, meaning that
    the morphism \(\taylor \colon \germalg n \to \formser \real X\)
    which sends a germ to its Taylor series at \(0\) is injective. 
  \end{itemize}
\end{thm}
Then, \cite[Section 3]{legal-rolin-not-c-infty} is dedicated to
proving that the structure \(\struct = \langle \real_\text{alg}; H,
H', H'', \dots \rangle\),
which is the expansion of the real ordered field with function symbols
for the derivatives of any order order of \(H\), is o-minimal and
model-complete. This structure clearly does not have \(C^\infty\)
cell-decomposition so that the result of
\cite{rolin-servi-monomialization} does not apply in this
case. Despite that, the proofs of o-minimality and model-completeness
given in \cite[Section 3]{rolin-servi-monomialization} 
and in \cite[Section 3]{legal-rolin-not-c-infty} are rather similar
and rely heavily on a Fiber Cutting Lemma.
The main difference between the two is that some global results in
\cite{rolin-servi-monomialization} have to be stated locally in
\cite{legal-rolin-not-c-infty}. However, the local methods used in
\cite{legal-rolin-not-c-infty} involve restricting to small
neighborhoods of \(0\) often and,
as shown in an example given in the Appendix to this text, their proof
of Lemma 3.7 sometimes constructs sets with empty germ meaning that
such sets cannot always be restricted.
The goal of the
present document is to provide a framework that lets us restrict to
smaller neighborhoods of \(0\) at will, and hence produce an
alternative proof of the o-minimality and model-completeness of
\(\struct\).  

In a
forthcoming paper, we weaken the hypotheses of
\cite{rolin-servi-monomialization} to prove a very general
o-minimality and quantifier elimination result for structures
generated by generalized quasianalytic classes of weakly \(C^\infty\)
germs. In particular, given \(H\) as in Theorem
\ref{thm:properties-H}, there is an expansion of \(\struct\) with
existentially definable functions that satisfies these
weaker hypotheses, so that we obtain a quantifier elimination theorem
for an o-minimal structure without smooth cell-decomposition. \\

We now describe the plan of the present document. We fix
\(H \colon \real \to 
\real\) a function as in Theorem \ref{thm:properties-H}.
Roughtly, a \textbf{\(\Lambda\)-set} \(A \subset \real^n\) is a
bounded and quantifier-free definable subset (see Definition \ref{defn:H-basic}).
A \textbf{sub-\(\Lambda\)-set} is defined as a projection of a
\(\Lambda\)-set. Thus, in particular, sub-\(\Lambda\)-sets are
existentially definable. Sub-\(\Lambda\)-sets are easily seen to be
stable under finite unions and intersections as well as
projections. Provided we show that sub-\(\Lambda\)-sets are stable
under set-theoretic difference, we can establish a bijective
correspondence between \(\struct\)-definable sets and
sub-\(\Lambda\)-sets. Model-completeness is then immediate while
o-minimality follows if we can prove that \(\Lambda\)-sets have
finitely many connected components since projections preserve this
property. 

In Section \ref{sec:simple-sub-sets}, we introduce the notion of
simple sub-\(\Lambda\)-sets which is at the heart of this new approach
to the proof of o-minimality. In Theorem
\ref{basic-simple-sub-lambda-set}, we show that germs of
quantifier-free definable sets at \(0\) (namely germs of \(H\)-sets,
see Definition \ref{defn:H-basic}) have representatives that are also
simple sub-\(\Lambda\)-sets. This is done by showing inductively on
the structure of germs \(f \in \germalg n\) (see Definition
\ref{defn:smallest-algebras}) that the graph of \(f\) 
has a representative which is also a simple
sub-\(\Lambda\)-set. 

In Section \ref{sec:parametrization}, we prove a monomialization
result (Theorem \ref{*-monomialization}) from which we derive several
parametrization results by manifolds for \(H\)-sets. In
particular, we show in Corollary \ref{cor:lambda-finite-components}
that \(\Lambda\)-sets have finitely many connected components which
will later yield o-minimality. Monomialization is a process by which
we can replace a germ \(f \in \germalg n\) with a family of normal
germs, namely germs of the form \(x^\alpha u(x)\) where \(u
\in \germalg n\) is such that \(u(0) \neq 0\). The main advantage of
such germs is that their sign depends only on the signs of the various
coordinates.

In Section \ref{cutting-fibers}, we show that every
sub-\(\Lambda\)-set is simple (see Theorem
\ref{sub-sets-are-simple}). Combining this with Lemma
\ref{global-parametrization}, we see that the hypotheses of
Gabrielov's Theorem of the Complement \cite[Theorem
2.7]{vdDries-Speissegger-generalized-power-series} are
satisfied. Finally, in Theorem \ref{thm:omin}, we establish a
bijection between \(\struct\)-definable sets and sub-\(\Lambda\)-sets,
allowing us to deduce model-completeness and o-minimality at once.
The main novelty of this approach has to do with the Fiber Cutting
Lemma, of which there are now two
versions, a local one (Proposition \ref{local-fiber-cutting}) and a
global one (Corollary \ref{global-fiber-cutting}). The classical proof
relies on \(\mathcal{A}\)-analyticity \cite[Definition
1.10]{rolin-servi-monomialization}, which is a global assumption that
does not hold in the present case.
Finally, the importance of simple sub-\(\Lambda\)-sets is stressed
throughout Paragraph \ref{par:global-fiber-cutting}, as they can be
replaced seamlessly by \(\Lambda\)-sets. \\

\section{A local property of \(H\)-basic sets}
\label{sec:simple-sub-sets}

Throughout the text, we fix a function \(H\) such as in Theorem
\ref{thm:properties-H} and we let \(\struct\) be the structure
\(\langle \real_\text{alg}; H, H', H'', \dots \rangle\). Notice that
\(\struct\) and \(\langle \real_\text{alg}; H\rangle\) have
the same definable sets so that one is o-minimal if and only if the
other is as well. However, \(\struct\) is model-complete while
\(\langle \real_\text{alg}; H \rangle\) is not. We start by recalling
a few definitions.

\begin{defn}
  \label{defn:H-basic}
  Firstly, if \(r = (r_1, \dots, r_n) \in \real^n\), then we define
  the \textbf{polydisk}
  \(\polydisk n r = (-r_1, r_1) \times \dots \times (-r_n,
  r_n)\). Given a continuous function
  \(f \colon \polydisk n r \to \real\), we write \(f \in \germalg n\)
  to mean that the germ at \(0\) of \(f\) is in \(\germalg n\). A
  subset \(A \subset \polydisk n r\) is said to be
  \textbf{\(H\)-basic} if there are
  \(g_0, g_1, \dots, g_q \in \germalg n\) such that
  \[A = \{x \in \polydisk n r \for g_0(x) = 0, g_1(x) > 0, \dots,
    g_q(x) > 0\}.\] An \textbf{\(H\)-set} is a finite union of
  \(H\)-basic sets. A set \(A \subset \real^n\) is said to be
  \textbf{\(H\)-semianalytic} when, for any \(a \in \real^n\), there
  is a polydisk \(\polydisk n r\) such that
  \((A-a) \cap \polydisk n r\) is an \(H\)-set. If
  \(A \subset \real^n\) is both \(H\)-semianalytic and bounded we say
  that \(A\) is a \textbf{\(\Lambda\)-set}. Given integers
  \(n, k \geq 0\), we let \(\Pi_n \colon \real^{n+k} \to \real^n\) be
  the projection on the first \(n\) variables. Then, if
  \(A \subset \real^n\) is such that there are an integer \(k \geq 0\)
  and a \(\Lambda \)-set \(A' \subset \real^{n+k}\) with
  \(\Pi_n(A') = A\), we say that \(A\) is a
  \textbf{sub-\(\Lambda\)-set}.
\end{defn}
In the defininition of \(H\)-basic sets, the condition \(g_0, \dots, q_q \in
\germalg n\) is really only a condition on the germs of \(g_0, \dots,
g_q\) at \(0\). Thus, \(H\)-basic sets can be quite arbitrary away
from \(0\), meaning that the definition of \(H\)-basic sets is local
around \(0\). On the
contrary, if \(A \subset \real^n\) is \(H\)-semianalytic
then we have information on the germ of \(A\) at every point \(a \in
\real^n\). Thus, this latter definition is global. As a first step
towards proving o-minimality, we would like to show that there is a
link between the two notions. In \cite{rolin-servi-monomialization},
such a link is obtained by forcing the functions \(g_0, g_1, \dots,
g_q\) that appear in the definition of \(H\)-basic sets to
be \(H\)-analytic, meaning that their germ at every point \(a \in
\polydisk n r\) is the translation at \(a\) of a germ in
\(\germalg n\) \cite[Definition
1.10]{rolin-servi-monomialization}. However, this approach is valid
only as long as each germ in \(\germalg n\) has an
\(H\)-analytic representative. But, such a representative would have
partial derivatives of every order showing that it is
\(C^\infty\). Since \(H\) has no \(C^\infty\)
representative, this method is unavailable in our case. Instead, we
introduce the notion of simple sub-\(\Lambda\)-sets, which are a
special instance of sub-\(\Lambda\)-sets, and we show that, for every
\(H\)-set \(A \subset \real^n\), there exists a polydisk
\(\polydisk n r\) such that \(A \cap \polydisk n r\) is a simple
sub-\(\Lambda\)-set.

\subsection{Simple sub-\(\Lambda\)-sets}

\begin{defn}
  \label{defn:simple-sub-sets}
  A sub-\(\Lambda\)-set \(A \subset \real^n\) is called
  \textbf{simple} when it is existentially definable in \(\struct\)
  and when there
  are an integer \(k \geq 0\) and a \(\Lambda\)-set \(A' \subset
  \real^{n+k}\) such that \(\Pi_n(A') = A\) and
  the fibers of \(\Pi_n \restriction A'\) are finite, namely, for
  every \(a \in A\), the set
  \[A'_a \coloneq \{b \in \real^k \for (a, b) \in A'\}\]
  is finite.
\end{defn}

In the notation of the definition, we will show in Proposition
\ref{dimension-fibers} that
\(\dim(A) = \dim(A')\) for a reasonable definition of dimension. This,
along with the following set theoretical stability properties of simple
sub-\(\Lambda\)-sets, explain the usefulness of the
notion. 

\begin{prp}
  \label{simple-prp}
  Let \(A, B \subset \real^n\) and \(C \subset \real^m\) be
  simple sub-\(\Lambda\)-sets. Then, the sets \(A \cap B, A \cup B\)
  and \(A \times C\) are also simple sub-\(\Lambda\)-sets. If \(k \leq
  n\) and \(x \in \real^k\) then the fiber
  \[A_x = \{y \in \real^{n-k} \for (x, y) \in A\}\]
  is a simple sub-\(\Lambda\)-set. Furthermore, if \(\Pi_k
  \restriction A\) has finite fibers then 
  \(\Pi_k(A)\) is also a simple sub-\(\Lambda\)-set. Finally, if
  \(s(1), \dots, s(n) \in \{1, \dots, k\}\) then
  \[\{(x_1, \dots, x_k) \in \real^k \for (x_{s(1)}, \dots, x_{s(n)})
    \in A\}\]
  is a simple sub-\(\Lambda\)-set. 
\end{prp}

Every part of the proposition above is immediate so we omit the proof.

\subsection{The main result}

As announced in the introduction of this section, the main theorem we
are looking to prove is the following. 

\begin{thm}
  \label{basic-simple-sub-lambda-set}
  If \(A \subset \real^n\) is an \(H\)-set then there is a
  polydisk \(\polydisk n r\) such that \(A \cap \polydisk n r\) is a
  simple sub-\(\Lambda\)-set.
\end{thm}

The theorem is a consequence of the following lemma. 

\begin{lem}
  \label{graphs}
  Given \(f \colon \polydisk n r \to \real\) such that \(f \in
  \germalg n\), there is a polydisk \(\polydisk n{r'} \subset
  \polydisk n r\) such that \(\graph{f \restriction \polydisk n{r'}}\)
  is a simple sub-\(\Lambda\)-set.
\end{lem}

We start by proving the theorem using the lemma. 

\begin{proof}[Proof of Theorem \ref{basic-simple-sub-lambda-set}]
  Let \(g_0, g_1, \dots, g_q \in \germalg n\)
  be functions defined on a polydisk \(\polydisk n r\) such that
  \[A = \{x \in \polydisk n r \for g_0(x) = 0,\ g_1(x) > 0,\ \dots,\
    g_q(x) > 0\}.\]
  By the lemma, up to shrinking \(\polydisk n r\), we may assume
  that \(\graph{g_i}\) is a simple
  sub-\(\Lambda\)-set for
  each \(0 \leq i \leq q\). In particular, the functions \(g_0, \dots,
  g_q\) are bounded. Thus, there exists a polyradius \(s = (s_0,
  \dots, s_q)\) such that \(\abs {g_i(x)} < s_i\) for each \(0 \leq i
  \leq q\) and each \(x \in \polydisk n r\). Then, 
  \begin{align*}
    A' =
    &\{(x, y) \in \polydisk n r \times \polydisk {q+1} s \for y_i = g_i(x)
      \text{ for } 0 \leq i \leq q\} \\
    \cap & \{(x, y) \in \polydisk n r \times \polydisk {q+1} s \for y_0 =
           0,\ y_1 > 0,\ \dots,\ y_q > 0\}
  \end{align*}
  is a simple
  sub-\(\Lambda\)-set. Furthermore, \(A = \Pi_n(A')\) and
  the fibers of \(\Pi_n \restriction A'\) are finite. 
\end{proof}

The aim of the rest of this section is to prove the lemma. Given the
definition of \(\germalg n\) (see Definition
\ref{defn:smallest-algebras}), the most
natural way to prove the lemma would be as follows. First define
\(\mathcal{B}_n\) to be the set of germs \(f \in \mathcal{W}_n\)
having a representative \(f \colon \polydisk n r \to \real\) such that
\(\graph f\) is a simple
sub-\(\Lambda\)-set. We have to show that
\(\mathcal{A}_n(H) \subset \mathcal{B}_n\). Since \(\mathcal{B}_n\) is
a collection of algebras stable under composition and implicit
functions, it suffices to prove that it is also stable under monomial
division. However, there does not seem to be an easy way of showing
this. Thus, we have to inductively prove a stronger property. 

\begin{defn}
  Consider an open set \(U \subset \real^n\), a function \(f \colon U \to
  \real\) and an integer \(p \geq 0\). We say that \(f\) has property \((*_p)\)
  if \(f\) is \(C^p\) and,
  for any multi-index \(\alpha\) such that \(\length \alpha \leq p\),
  the graph of \(\frac{\partial^{\length \alpha} f}{\partial
    x^\alpha}\) is a simple sub-\(\Lambda\)-set. 
\end{defn}

\begin{remark}
  If \(f \colon \real^n \to \real\) is a polynomial and \(\polydisk n r\) is a
  polydisk, then the graph of \(f \restriction \polydisk n r\) is a
  quantifier-free definable \(\Lambda\)-set. Thus, \(f \restriction
  \polydisk n r\) has
  \((*_p)\) for every \(p \geq 0\). Also, since the germ of \(H\) at
  \(0\) is weakly \(C^\infty\), for any integer \(p \geq 0\), there
  exists an interval \(\polydisk 1 r\) such that \(H \restriction \polydisk 1 r\) is
  \(C^p\). Furthermore, since \(H\) is piecewise a polynomial in the
  complement of any neighborhood of \(0\), the graph of the derivative
  \(H^{(k)} \restriction \polydisk 1 r\) is a \(\Lambda\)-set for
  every \(k \leq p\). In
  particular, for any \(p \geq 0\), there is some polyradius \(r\)
  such that \(H \restriction \polydisk 1 r\) has \((*_p)\). 
\end{remark}

Assuming the results proven in the next two paragraphs, we can now
prove the lemma as follows.

\begin{proof}[Proof of Lemma \ref{graphs}]
  We define \(\mathcal{B}_n\) to be the set of germs \(f \in
  \mathcal{W}_n\) such that, for every \(p \geq 0\), \(f\) has a
  reprensentative \(f \colon \polydisk n r \to \real\) having
  \((*_p)\). It is clear that \(\mathcal{B}_n\) is a subalgebra of
  \(\mathcal{W}_n\).
  By the remark above, \(H \in \mathcal{B}_1\) and the
  germs of polynomials in \(n\) variables belong to
  \(\mathcal{B}_n\). Furthermore, by Propositions \ref{composition},
  \ref{implicit-function} and \ref{monomial-division}, this
  collection of algebras is stable under composition,
  implicit functions and monomial division respectively. Thus,
  \(\mathcal{A}_n(H) \subset \mathcal{B}_n\)  by Definition
  \ref{defn:smallest-algebras} whence the result. 
\end{proof}

\subsection{Composition and implicit functions}

The goal of this paragraph is to show that the algebras
\(\mathcal{B}_n\) in the proof of Lemma \ref{graphs} are stable under
composition of functions and taking implicit functions.
A rather natural approach to these results involves proving stability
properties for functions having the property \((*_p)\) for a fixed
integer \(p\), which is more than we need. 
In this paragraph and the
next, if \(a = (a_1, \dots, a_n)\), it will be convenient to write 
\(\widehat a = (a_1, \dots, a_{n-1})\). Since both of the propositions
below are proven in essentially the same way, we omit the proof of
Proposition \ref{composition} about composition. 

\begin{prp}
  \label{composition}
  Let \(U \subset \real^n\) and \(V \subset \real^m\) be open sets and
  consider functions \(f_1, \dots, f_m \colon U \to \real\) and \(g
  \colon V \to \real\) that all have \((*_p)\). Assume that
  \(f(U) \subset V\) where \(f = (f_1, \dots, f_m)\). Then, the
  function \(g \circ f\) has \((*_p)\). 
\end{prp}

\begin{prp}
  \label{implicit-function}
  Let \(U \subset \real^n\) be an open set and consider a \(C^1\)
  function \(f
  \colon U \to \real\) that has \((*_p)\). Consider also \(a \in U\)
  such that \(f(a) = 0\) and \(\frac{\partial f}{\partial x_n}(a) \neq
  0\). Then, there
  exists an open neighborhood \(V \subset \real^{n-1}\) of \(\widehat
  a\) and a function \(\varphi \colon V \to \real\) having \((*_p)\) such
  that \(\varphi(\widehat a) =
  a_n\) and \(f(\widehat x, \varphi(\widehat x)) = 0\) for all \(\widehat x
  \in V\).
\end{prp}

\begin{proof}
  We prove the result by induction on \(p\). Firstly, assume that \(p
  = 0\). Since \(f\) is \(C^1\), the Implicit Functions Theorem yields
  a polydisk \(\polydisk n r\) such that
  \(a + \polydisk n r \subset U\) and, for each \(\widehat x \in
  \widehat a + \polydisk {n-1}{\widehat r}\), there is a unique \(x_n\)
  such that \(x \in a + \polydisk n r\) and \(f(x) = 0\). The set
  \[A = \{(x, y) \in \real^n \times \real \for x \in (a + \polydisk n r),\
    y = f(x) \text{ and } y=0\}\]
  is a simple sub-\(\Lambda\)-set. The
  projection of \(A\) on the
  \(x\)-coordinates is the simple sub-\(\Lambda\)-set
  \[B = \{x \in a + \polydisk n r \for f(x) = 0\}.\]
  The results above show that the set \(B\) is the graph of a function
  \(\varphi \colon V \to \real\) where \(V = \widehat a + \polydisk
  {n-1} {\widehat r}\). Since the fibers of this projection are
  finite, the graph of \(\varphi\) is a simple
  sub-\(\Lambda\)-set. Finally, we have \(f(\widehat x,
  \varphi(\widehat x)) = 0\) for all \(\widehat x \in V\) by
  definition.

  Before we consider the case \(p > 0\), we need an intermediate
  result. Consider \(0 < \alpha < \beta\) and let \(U = \{x \in \real \for
  \alpha < \abs x < \beta\}\). Define \(g \colon U \to \real\) as
  \(g(x) = \frac{1}{x}\). We are going to show that \(g\) has
  \((*_p)\) for every \(p \geq 0\). Firstly, if \(r = \left(\beta,
    \frac{1}{\alpha}\right)\), then 
  \[\Gamma(g) = \{(x, y) \in \polydisk n r \for xy = 1\}\]
  is a \(\Lambda\)-set. Also, given \(p \geq 0\), the derivative
  \(g^{(p)}\) has the form \(g^{(p)} = P \circ g\) where \(P\) is a
  polynomial. By composition, we deduce that the graph of \(g^{(p)}\)
  is a simple sub-\(\Lambda\)-set which concludes.

  Going back to the proof of the proposition, assume that \(p >
  0\). Then, we already know that there is an open neighborhood \(V \subset
  \real^{n-1}\) of \(\widehat a\) and a function \(\varphi \colon V \to \real\) having
  \((*_{p-1})\) such that \(\varphi(\widehat a) = a_n\) and \(f(\widehat x,
  \varphi(\widehat x)) = 0\) for all \(\widehat x \in V\). Up to
  shrinking \(V\), we may also assume that there are \(0 < \alpha <
  \beta\) such that \(\alpha < \abs{\frac{\partial f}{\partial
      x_n}(\widehat x, \phi(\widehat x))} < \beta\) for all \(\widehat
  x \in V\).
  Then, given \(1 \leq i \leq n-1\) and \(\widehat x \in V\), we
  have
  \[\frac{\partial \varphi}{\partial x_i}(\widehat x) = -\frac{\frac{\partial
        f}{\partial x_i}(\widehat x, \varphi(\widehat x))}{\frac{\partial
        f}{\partial x_n}(\widehat x, \varphi(\widehat x))}.\]
  By assumption, we know that the functions \(\frac{\partial
    f}{\partial x_i}\) and \(\frac{\partial f}{\partial x_n}\) have
  \((*_{p-1})\). Furthermore, by the inductive hypothesis, 
  \(\varphi\) also has \((*_{p-1})\). Thus, using
  Proposition \ref{composition} about composition, the fact that the
  polynomial \((x, y) \mapsto xy\) has \((*_{p-1})\) and the above
  result about the function \(g\), we find 
  that \(\frac{\partial \varphi}{\partial x_i}\) has \((*_{p-1})\). Since
  this holds for all \(1 \leq i \leq n-1\), the function \(\varphi\) must
  have \((*_p)\). 
\end{proof}

\subsection{Monomial division}

This paragraph is devoted to showing that the algebras
\(\mathcal{B}_n\) introduced in the proof of Lemma \ref{graphs} are
stable under monomial division. The proof of this result is more
delicate than those of the corresponding results for composition and
implicit functions. This is because there does not seem to be an easy
way to argue that functions having property \((*_p)\) are closed under
monomial division. 

\begin{prp}
  \label{monomial-division}
  Let \(\polydisk n r\) be a polydisk and consider a function \(f \colon
  \polydisk n r \to \real\). Assume that, for every \(p \geq 0\),
  there is some polydisk \(\polydisk n {r'} \subset \polydisk n r\)
  such that \(f \restriction \polydisk n {r'}\) has \((*_p)\).
  Consider also \(q \geq 0\) such that, for each \(0 \leq k \leq q\),
  the derivative \(\frac{\partial^k f}{\partial x_n^k}\) exists and verifies
  that \(\frac{\partial^k f}{\partial x_n^k}(x) = 0\) for all \(x \in
  \polydisk n r\) such that 
  \(x_n = 0\). Then, for every \(p \geq 0\), there is some polydisk
  \(\polydisk n {r'} \subset \polydisk n r\) such that there is a
  continuous function \(g \colon \polydisk n {r'} \to \real\) having
  \((*_p)\) with
  \(x_n^{q+1}g(x) = f(x)\) for every \(x \in \polydisk n {r'}\). 
\end{prp}

Notice in particular that the assumptions of the lemma imply that the
germ of \(f\) at \(0\) is weakly \(C^\infty\). 

\begin{proof}
  We prove by induction on \(p\) that the result holds for all
  integers \(q\) and all functions \(f\) with the required
  properties. Firstly, assume that
  \(p=0\). Up to shrinking \(\polydisk n r\), we may assume that
  \(f\) has \((*_{q+1})\). In particular, since \(f\) is \(C^{q+1}\),
  we can extend \(x \mapsto \frac{f(x)}{x_n^{q+1}}\) to a continuous
  function \(g \colon \polydisk n r \to \real\) such that
  \[x_n^{q+1}g(x) = f(x)\]
  for every \(x \in \polydisk n r\). Since
  \(f\) has \((*_0)\), the set
  \[\{(x, y, z) \in \real^n \times \real \times \real \for x_n \neq 0,\ y = f(x),\
    zx_n^{q+1} = y\}\]
  is a simple sub-\(\Lambda\)-set. Also, the projection of this set on
  the \((x, z)\)-coordinates is the graph of \(g \restriction \{x \in
  \polydisk n r \for x_n \neq 0\}\). Furthermore, if \(x \in
  \polydisk n r\) is such that \(x_n = 0\), then we have
  \[g(x) = \frac{1}{(q+1)!} \frac{\partial^{q+1} f}{\partial
      x_n^{q+1}}(x).\]
  Also, since \(f\) has \((*)_{q+1}\), it follows that
  \[\left\{(x, y) \in \real^n \times \real \for x_n = 0,\ y =
    \frac{1}{(q+1)!}\frac{\partial^{q+1}f}{\partial x_n^{q+1}}(x)\right\}\]
  is a simple sub-\(\Lambda\)-set. Since this is the graph of \(g
  \restriction \{x \in \polydisk n r \for x_n = 0\}\), we see that the
  graph of \(g\) is a simple sub-\(\Lambda\)-set by taking a union. 

  Now, assume that \(p > 0\). Up to shrinking \(\polydisk n r\),
  we can assume that there is a function
  \(g \colon \polydisk n r \to \real\) that has \((*_0)\) and such that
  \[x_n^{q+1}g(x) = f(x)\]
  for every \(x \in \polydisk n r\) such that \(x_n \neq 0\). Up to
  shrinking \(\polydisk n r\) some more, we can assume that \(f\) is
  \(C^{q+2}\). Next, given an integer \(1 \leq i \leq n-1\), 
  \(\frac{\partial g}{\partial x_i}\) is the continuous
  function such that
  \[x_n^{q+1}\frac{\partial g}{\partial x_i}(x) = \frac{\partial
        f}{\partial x_i}(x)\]
  for every \(x \in \polydisk n r\). By the
  inductive hypothesis, up to shrinking \(\polydisk n r\) finitely
  many times, we are reduced to the case when
  \(\frac{\partial g}{\partial x_i}\) has \((*_{p-1})\)
  for each \(1 \leq i \leq n-1\). Finally, \(\frac{\partial g}{\partial x_n}\) is the continuous
  function such that
  \[x_n^{q+2} \frac{\partial g}{\partial x_n}(x) = x_n\frac{\partial
      f}{\partial x_n}(x) - (q+1) f(x)\]
  for every \(x \in \polydisk n r\). 
  Thus, by the inductive hypothesis and Proposition \ref{composition},
  we can assume that \(\frac{\partial g}{\partial x_n}\) has
  \((*_{p-1})\) up to shrinking \(\polydisk n r\). Finally, \(g\)
  has \((*_0)\)
  and, for every \(1 \leq i \leq n\), \(\frac{\partial g}{\partial
    x_i}\) has \((*_{p-1})\). Thus, \(g\) has \((*_p)\) whence the
  result. 
\end{proof}

\section{Parametrization results}
\label{sec:parametrization}

In this section we prove a local parametrization result for
\(H\)-basic sets around \(0\). In Paragraph \ref{par:monomialization},
we start by recalling the process of monomialization. This process
allows us to replace germs in \(\germalg n\) with normal germs,
namely, germs \(f \in \germalg n\) such that there are \(\alpha \in
\nat^n\) and \(u \in \germalg n\) with \(u(0) \neq 0\) and \(f =
x^\alpha u\). The advantage of such germs is that, in a neighborhood
of \(0\), the sign of \(f\) only depends on the signs of the variables
\(x_1, \dots, x_n\). In particular, \(H\)-basic sets defined by
equations and inequations involving only normal germs have a very
simple structure. In Paragraph \ref{par:first-parametrization}, we
take advantage of this to
obtain a first parametrization result, namely Proposition
\ref{parametrization-sub-quadrants}. We then introduce the notion of
\(H\)-manifolds which are both
manifolds and \(H\)-basic sets at the same time. We rephrase
Proposition \ref{parametrization-sub-quadrants} in terms of
\(H\)-manifolds in Corollary \ref{conj-local-parametrization}. The two
remaining paragraphs are dedicated to proving a refinement of this
latter result by showing that the manifolds involved in the
parametrization may verify two, rather technical, regularity conditions.

\subsection{Monomialization}
\label{par:monomialization}

This paragraph is a summary that aims at proving a slightly stronger
monomialization result than \cite[Theorem
2.11]{rolin-servi-monomialization}. We recall the notations but we
refer to \cite[Section 2]{rolin-servi-monomialization} for the proofs
of the results that do not require any modification.

\begin{defn}
  An \textbf{elementary transformation} is a map \(\nu \colon \polydisk n {r'} \to
  \polydisk n r\) of one of the following forms.
  \begin{itemize}
  \item \textbf{A blow-up chart}: Given \(1 \leq i, j \leq n\) with
    \(i \neq j\) and \(\lambda \in \real\), define
    \begin{equation*}
      \pi^\lambda_{i, j}(x') = x \text{ where } 
      \begin{cases}
        x_k = x_k' &k \neq i\\
        x_i = x_j'(\lambda + x_i').
      \end{cases}
    \end{equation*}
    Also, define \(\pi_{i, j}^\infty = \pi_{j, i}^0\). 
  \item \textbf{A Tschirnhausen translation}: Given \(h \in \germalg
    {n-1}\) such that \(h(0) = 0\), define
    \begin{equation*}
      \tau_h(x') = x \text{ where }
      \begin{cases}
        x_k = x_k' &k \neq n\\
        x_n = x_n' + h(x_1', \dots, x_{n-1}').
      \end{cases}
    \end{equation*}
  \item \textbf{A shear transformation}: Given \(1 \leq i \leq n\) and
    \(c_1, \dots, c_{i-1}\), define
    \begin{equation*}
      L_{i, c}(x') = x \text{ where }
      \begin{cases}
        x_k = x_k' &i \leq k \leq n\\
        x_k = x_k' + c_kx_i' &1 \leq k < i.
      \end{cases}
    \end{equation*}
  \item \textbf{A ramification}: Given \(1 \leq i \leq n\) and \(d \in
    \nat\), define
    \begin{equation*}
      r_i^{d, \pm}(x') = x \text{ where }
      \begin{cases}
        x_k = x_k' &k \neq i\\
        x_i = \pm x_i^{\prime d}
      \end{cases}
    \end{equation*}
  \end{itemize}
  An \textbf{admissible transformation} is a composition of elementary
  transformations. 
\end{defn}

\begin{remark}
  Let \(\rho\) be an admissible transformation. If we write \(\rho =
  (\rho_1, \dots, \rho_n)\) then \(\rho_1, \dots, \rho_n \in \germalg
  n\). In particular, the germs of \(\rho_1, \dots, \rho_n\) at \(0\)
  are weakly \(C^\infty\). Furthermore, \(\rho(0) = 0\) so that, for
  any series \(F \in \formser \real X\), we can define
  \[F \circ \rho = F(\taylor(\rho_1), \dots, \taylor(\rho_n))\]
  where \(\taylor \colon \germalg n \to \formser \real X\) is the map
  which takes a germ to its Taylor series. 
  It is easy to see that the map \(\formser \real X \to \formser \real
  X, F \mapsto F \circ \rho\) is an algebra homomorphism. Also, by
  composition of Taylor series, if \(f \in \germalg n\) then \(\taylor
  (f \circ \rho) = \taylor(f) \circ \rho\). 
\end{remark}

\begin{prp}[{\cite[Lemma 2.5]{rolin-servi-monomialization}}]
  If \(\rho\) is an admissible transformation then the algebra
  homomorphism \(F \mapsto F \circ \rho\) defined in the remark above
  is injective. 
\end{prp}

\begin{defn}
  An \textbf{elementary tree} is a tree of any of the following forms.
  \begin{itemize}
  \item Given \(1 \leq i, j \leq n\) such that \(i \neq j\), consider
    the tree
    
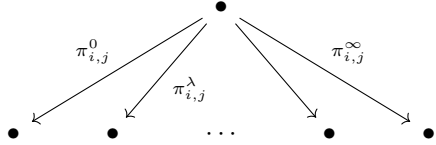
\begin{center}
      \begin{tikzcd}
        &&\bullet \arrow[']{ddll}{\pi_{i, j}^0} \arrow{ddl}{\pi_{i,
            j}^\lambda} \arrow{ddr} \arrow{ddrr}{\pi_{i, j}^\infty}\\
        \\
        \bullet &\bullet &\cdots &\bullet &\bullet
      \end{tikzcd}
    \end{center}
    with one branch for each transformation in the family \((\pi_{i,
      j}^\lambda)_{\lambda \in \real \cup \{\infty\}}\). 
  \item Given \(h \in \germalg {n-1}\) with \(h(0) = 0\), consider the
    tree
    \begin{center}
      \begin{tikzcd}
        \bullet \arrow{dd}{\tau_h}\\
        \\
        \bullet
      \end{tikzcd}
    \end{center}
  \item Given \(1 \leq i \leq n\) and \(c = (c_1, \dots, c_{i-1})\),
    consider the tree
    
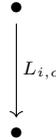
\begin{center}
      \begin{tikzcd}
        \bullet \arrow{dd}{L_{i, c}}\\
        \\
        \bullet
      \end{tikzcd}
    \end{center}    
  \item Given \(1 \leq i \leq n\) and \(d \in \nat\), consider the
    tree
    
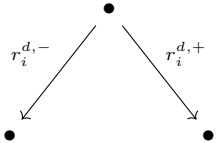
\begin{center}
      \begin{tikzcd}
        &\bullet \arrow[']{ddl}{r^{d, -}_i} \arrow{ddr}{r^{d,
            +}_i}\\
        \\
        \bullet &&\bullet
      \end{tikzcd}
    \end{center}
  \end{itemize}
  \textbf{Admissible trees} of height at most \(h\) are defined
  inductively on the ordinal \(h\). An
  admissible tree of height \(0\) is simply a vertex
  \(\bullet\). Given an ordinal \(h \geq 1\), an
  admissible tree of height at most \(h\) is an elementary tree with
  admissible trees of height \(<h\) attached to each of its
  leaves. 
\end{defn}

\begin{remark}
  Notice that the height of an admissible tree \(T\) is an ordinal so
  that  admissible trees may be infinite. However, each branch of
  \(T\) is finite so that it induces an admissible transformation
  \(\rho\). 
\end{remark}

\begin{defn}
  Let \(X = (X_1, \dots, X_n)\) be a tuple of variables and \(F
  \in \formser \real X\) be non-zero. We say that \(F\) is \textbf{normal} when
  there are \(\alpha \in \nat^n\) and \(U \in \formser \real X\) an
  invertible series, namely an invertible element of the ring
  \(\formser \real X\), such that \(F = X^\alpha U\). If \(f \in \germalg
  n\) is non-zero, we say that it is \textbf{normal} when
  \(\taylor(f)\) is normal. 
\end{defn}

\begin{remark}
  A series \(F \in \formser \real X\) is invertible if and only if
  \(F(0) \neq 0\). Also, consider \(f \in \germalg n\) a normal
  germ. Then, there are \(\alpha \in \nat^n\) and \(U \in \formser
  \real X\) an invertible series such that \(\taylor(f) = X^\alpha
  U\). But then, since germs in \(\germalg n\) are stable under monomial
  division, it follows that there is a germ \(u \in \germalg n\) such
  that \(f(x) = x^\alpha u(x)\). It is now clear that \(\taylor(u) =
  U\) whence, in particular, \(u(0) \neq 0\). 
\end{remark}

\begin{lem}[{\cite[Lemma 2.9]{rolin-servi-monomialization}}]
  \label{monomialization-product}
  If \(F_1, \dots, F_p \in \formser \real X\) are non-zero series then
  they are all normal if and only if their product \(\prod_{i=1}^p
  F_i\) is normal. Also, if \(F, G \in \formser \real X\) are two
  non-zero series such that \(F, G\) and \(F-G\) are all normal then
  either \(F \mid G\) or \(G \mid F\). 
\end{lem}

\begin{defn}
  Given \(F_1, \dots, F_p \in \ima \taylor\) and an admissible tree
  \(T\), we say that \(T\) \textbf{monomializes} \(F_1, \dots, F_p\) when
  \(F_1 \circ \rho, \dots, F_p \circ \rho\) are normal for each
  admissible transformation \(\rho\) induced by a branch of \(T\).

  Given a blow-up chart \(\nu = \pi_{i, j}^\lambda\) with \(\lambda
  \neq \infty\), we have
  \begin{equation*}
    \pi_{i, j}^\lambda(X') = X \text{ if and only if }
    \begin{cases}
      X'_k = X_k &k \neq i\\
      X'_i = \frac{X_i}{X_j} - \lambda
    \end{cases}
  \end{equation*}
  Thus, we need to divide by \(X_j\) in order to write the inverse of
  \(\nu\). We call \(X_j\) the \textbf{critical variable} of
  \(\nu\). Since \(\pi_{i, j}^\infty = \pi_{j, i}^0\), this definition
  applies to all blow-up charts. We say that \(T\)
  \textbf{\(*\)-monomializes} \(F_1, \dots, F_p\) whenever it monomializes
  \(F_1, \dots, F_p\) and, for any blow-up chart \(\nu\) in \(T\), the
  sub-tree \(T'\) of \(T\) below \(\nu\) monomializes \(W\), where
  \(W\) is the critical variable of \(\nu\).
\end{defn}

\begin{remark}
  Given an elementary transformation \(\nu\) and \(F \in \ima \taylor\),
  we let \(\nu^*(F) = W \cdot (F \circ \nu)\) when \(\nu\) is a blow-up chart
  with critical variable \(W\) and \(\nu^*(F) = F \circ \nu\)
  otherwise. In view of Lemma \ref{monomialization-product}, an
  admissible tree \(T\)
  \(*\)-monomializes \(F\) if and only if, for each branch \((\nu_1,
  \dots, \nu_h)\) of \(T\), the germ \(\nu_h^* \circ \dots \circ
  \nu_1^*(F)\) is normal. 
\end{remark}

\begin{thm}
  \label{*-monomialization}
  Let \(F_1, \dots, F_p \in \ima \taylor\) be non-zero series. There is
  an admissible tree \(T\) that \(*\)-monomializes \(F_1, \dots, F_p\).
\end{thm}

Similar, but slightly different statements can be found in
\cite[Theorem 2.11]{rolin-servi-monomialization} or \cite[Section
2]{rolin-speissegger-wilkie-denjoy-carleman-classes} for
instance. Nonetheless, we include the proof for clarity, and also
because we want to insist on the need to keep the variables \(X_1,
\dots, X_n\) monomialized after every step of the algorithm, as
explained in the remark below. 

\begin{remark}
  Suppose that we have series \(F, G \in \formser \real X\) such that
  \(F = X_n G\). It is not necessarily true that any tree \(T\) which
  monomializes \(G\) will also monomialize \(F\). For instance,
  suppose \(n=2\) and write the variables \(X, Y\) instead of \(X_1,
  X_2\). If \(G = Y - X\) then \(G \circ \tau_X(X', Y') = Y'\) is
  normal but \(X \circ \tau_X(X', Y') = Y' + X'\) is not.

  However, assume that \(F, G \in \formser \real X\) and \(H \in
  \formser \real {\widehat X}\) are such that \(F = HG\). Assume also
  that \(H\) is normal and that \(T\) is an admissible tree satisfying
  the following conditions:
  \begin{itemize}
  \item There are no shear transformations in \(T\);
  \item The only Tschirnhausen translations in \(T\) act on the
    variable \(X_n\);
  \item The only blow-up charts of the form \(\pi_{n, i}^\infty\) with
    \(1 \leq i \leq n-1\) in \(T\) are at the leaves and there are no
    blow-ups of the form \(\pi_{i, n}\) with \(1 \leq i \leq n-1\) in
    the tree. 
  \end{itemize}
  It is easy to see that \(H \circ \rho\) is still normal, whenever
  \(\rho\) is an admissible transformation induced by a branch of
  \(T\). In particular, if \(T\) monomializes \(G\), then it also
  monomializes \(F\).

  In particular, it is very important that we are done with no further
  steps whenever we apply a blow-up chart of the form \(\pi_{n,
    i}^\infty\) as in case 1 of the proof below. This is also the
  reason why we pay attention throughout the proof at keeping the
  variables \(X_1, \dots, X_{n-1}\) monomialized at every step. The
  additional benefit this provides is that it guarantees that the tree
  \(T\) automatically \(*\)-monomializes the series \(F_1, \dots, F_p\)
  since the only blow-up charts that do not appear at the leaves of
  the tree have one of \(X_1, \dots, X_{n-1}\) as their critical
  variable.
\end{remark}

\begin{proof}
  We are going to prove the result by induction on \(n\), the case
  \(n=1\) being obvious.
  To begin with, we may assume that \(p=1\) by Lemma
  \ref{monomialization-product} and we write \(F = F_1 = \sum_{\alpha}
  a_\alpha X^\alpha\).
  Take \(c_1, \dots, c_{n-1} \in \real\) and let \(G = F \circ
  L_c\). We then have
  \[G(0, \dots, 0, X_n') = F(c_1X_n', \dots, c_{n-1}X_n', X_n') =
    \sum_{\alpha} a_\alpha c^{\widehat \alpha}
    X_n^{\prime \length \alpha} = \sum_{k \in \nat} Q_k(c) X_n^{\prime
      \length \alpha}\]
  where \(\widehat \alpha = (\alpha_1, \dots, \alpha_{n-1})\), and
  \[Q_k(Y_1, \dots, Y_{n-1}) = \sum_{\length \alpha = k} a_\alpha
    Y^{\widehat \alpha}.\]
  We also have
  \[F(X_1'X_n', \dots, X_{n-1}'X_n', X_n') = \sum_{k \in \nat}
    Q_k(\widehat{X'}) X_n^{\prime \length \alpha}\]
  and \(F(X_1'X_n', \dots, X_{n-1}'X_n', X_n')\) is obtained by
  composing \(F\) with blow-up charts. Since composition with blow-up
  charts is injective and \(F \neq 0\), we deduce that there must be
  some \(k \in \nat\) with \(Q_k \neq 0\). But then, there are \(c_1,
  \dots, c_{n-1}\in \real\) such that \(Q_k(c)\neq 0\). For this tuple
  \(c = (c_1, \dots, c_{n-1})\), we deduce that \(G(0, \dots, 0, X_n')
  \neq 0\), from which it follows that the series \(G\) is regular of
  some order \(d\) in
  \(X_n'\), namely, there are series \(G_1, \dots, G_d \in \formser
  \real {\widehat {X'}}\) and a unit \(U \in \formser \real {X'}\) such that
  \[G(X') = U(X')X_n^{\prime d} + G_1(\widehat {X'})X_n^{\prime d-1} +
    \dots + G_d(\widehat {X'})\]
  where \(\widehat {X'} = (X_1', \dots, X'_{n-1})\). It suffices to show
  that we can \(*\)-monomialize \(G\) so that it suffices to prove the
  result when \(F = G\). In this latter case, we write
  \[F(X) = U(X)X_n^d + F_1(\widehat X)X_n^{d-1} + \dots + F_d(\widehat
    X).\]  

  We now show by induction on \(d\) that we can
  \(*\)-monomialize \(F, X_1, \dots, X_{n-1}\), the result being
  obvious when \(d=0\). Assume that \(d > 0\) and let \(f \in \germalg
  n\) be such that \(F = \taylor(f)\). We have
    \[\taylor\left(\frac{\partial^d f}{\partial x_n^d}\right) =
        \frac{\partial^d F}{\partial X_n^d} =  d!U\] 
  so that \(\frac{\partial^d f}{\partial x_n^d}(0) \neq
  0\). Furthermore,
  \[\taylor\left(\frac{\partial^{d-1}f}{\partial
        x_n^{d-1}}\right) = \frac{\partial^{d-1} F}{\partial X_n^{d-1}} =
    (d-1)![F_1(\widehat X) + U(X)X_n].\]
  If \(F_1(0) \neq 0\) then \(F\) is regular in \(X_n\) of order
  \(d-1\) and the result follows by induction. Otherwise,
  \(\frac{\partial^{d-1} f}{\partial x_n^{d-1}}(0) = 0\). Thus, by
  Definition \ref{defn:smallest-algebras}, there is \(b \in \germalg
  {n-1}\) such that
  \(\frac{\partial^{d-1} f}{\partial x_n^{d-1}}(\widehat x, b(\widehat
  x)) = 0\) and \(b(0) = 0\). Now, if \(g = f \circ \tau_b\), we have
  \[\frac{\partial^{d-1} g}{\partial x_n^{\prime d-1}}(\widehat {x'}, 0) =
    \frac{\partial^{d-1} f}{\partial x_n^{d-1}}(\widehat {x'},
    b(\widehat {x'})) = 0.\]
  It is also easy to see that \(\taylor(g)\) is regular of order \(d\)
  in \(X_n\) so that, up to replacing \(F\) with \(F \circ \tau_b\),
  we can
  assume that \(\frac{\partial^{d-1} F}{\partial X_n^{d-1}}(\widehat
  x, 0) = 0\) which implies that \(F_1 = 0\). In particular, notice
  that \(F\) is then normal whenever \(d=1\).\\

  By the inductive hypothesis on \(n\), there is an admissible tree
  \(T\) that \(*\)-monomializes the series \(F_2, \dots, F_d, X_1,
  \dots, X_{n-1}\). Let \((\nu_1, \dots, \nu_k)\) be a branch of
  \(T\). We must now show that we can \(*\)-monomialize the
  series
  \begin{itemize}
  \item \(F \circ \nu_1 \circ \dots \circ \nu_k\);
  \item \(X_i \circ \nu_1 \circ \dots \circ \nu_k\) for every \(1 \leq
    i \leq n-1\);
  \item \(W \circ \nu_{j+1} \circ \dots \circ \nu_k\) whenever
    \(\nu_j\) is a blow-up chart with critical variable \(W\).
  \end{itemize}
  Since the series in the last two points are already normal, it
  suffices to show that we can \(*\)-monomialize \(F \circ \nu_1 \circ
  \dots \circ \nu_k\) and \(X_1, \dots, X_{n-1}\). Thus, up to
  replacing \(F\) with \(F \circ \nu_1 \circ \dots \circ \nu_k\), we
  can assume that \(F_2, \dots, F_d\) are all normal so that we can
  write \(F_i = \widehat X^{\alpha_i} U_i(\widehat X)\) where \(U_i\)
  is a unit.

  Let \(\sigma = (\sigma_1, \dots, \sigma_{n-1}) \in \{+,
  -\}^{n-1}\) and define
  \[r^\sigma = r_1^{d!, \sigma_1} \circ \dots \circ r_{n-1}^{d!,
      \sigma_{n-1}}.\]
  Up to replacing \(F\) with \(F \circ r^\sigma\), it suffices to
  prove the result when
  \(\alpha_i\) is divisible by \(d!\) for every \(2 \leq i \leq
  d\). Now, by the inductive hypothesis on \(n\)  there is an
  admissible tree \(T\)
  that \(*\)-monomializes \(X_1, \dots, X_{n-1}\) and the series
  \(\widehat X^{\frac{\alpha_i}{i}} - \widehat
  X^{\frac{\alpha_j}{j}}\) for every \(2 \leq i, j \leq d\).
  By the same argument as above, it suffices to show that,
  for every branch \((\nu_1, \dots, \nu_k)\) of \(T\), we can
  \(*\)-monomialize the series \(F \circ \nu_1 \circ \circ \nu_k\) and
  \(X_1, \dots, X_{n-1}\). Thus, by Lemma
  \ref{monomialization-product}, up to replacing \(F\) with \(F \circ
  \nu_1 \circ \dots \circ \nu_k\), we may assume that the tuples
  \(\frac{\alpha_i}{i}\) are linearly ordered. \\

  Let \(2 \leq l \leq d\) be minimal such that \(\frac{\alpha_l}{l}
  \leq \frac{\alpha_i}{i}\) for every \(2 \leq i \leq d\). Since the
  partial order on tuples \(\beta \in \nat^{n-1}\) is well-founded, we
  may prove the result by induction on \(\frac{\alpha_l}{l}\). When
  \(\alpha_l = 
  0\), \(F\) is regular in \(X_n\) of order \(d-l\) whence the result
  follows by induction. Otherwise, up to exchanging the order of the
  variables \(X_1, \dots, X_{n-1}\), we can assume that
  \(\alpha_{l, 1} \neq 0\). We can now do a blow-up between the
  variables \(X_1\) and \(X_n\). There are three cases to treat. For
  simplicity of notation, we still write \(X\) for the variables after
  the various blow-up charts.

  \textbf{Case 1:} \(\lambda = \infty\). 
  \begin{align*}
    F \circ \pi_{n, 1}^\infty (X)
    &= \widetilde U(X) X_n^d + \sum_{i=2}^d \widetilde {U_i}(X)
      \widehat X^{\alpha_i} X_n^{d-i+\alpha_{i, 1}} \\
    &= X^d_n\left[\widetilde U(X) + \sum_{i=2}^d \widetilde {U_i}(X)
      \widehat X^{\alpha_i} X_n^{\alpha_{i, 1} - i}\right]
  \end{align*}
  where \(\widetilde U = U \circ \pi_{n, 1}^\infty\) and
  \(\widetilde{U_i} = U_i \circ \pi_{n, 1}^\infty\) for every \(2 \leq
  i \leq d\). The series \(F \circ \pi_{n, 1}^\infty\) is normal since
  \(\widehat X^{\alpha_i}\) is never a constant for 
  \(2 \leq i \leq d\). Thus, we are done in this case.

  \textbf{Case 2:} \(\lambda \not \in \{\infty, 0\}\). 
  \begin{align*}
    F \circ \pi_{n, 1}^\lambda (X)
    &= \widetilde U(X) X_1^d(\lambda + X_n)^d + \sum_{i=2}^d
      \widetilde {U_i}(\widehat X)\widehat X^{\alpha_i}
       X_1^{d-i}(\lambda + X_n)^{d-i} \\
    &= X_1^d\left[ \widetilde U(X) X_n^d + d \lambda \widetilde U(X)
      X_n^{d-1}
      + \sum_{i=2}^d \binom d i \lambda^i \widetilde U(X) X_n^{d-i}
      \right.\\
      & \hspace{37pt}+ \left.\sum_{i=2}^d \widetilde{F_i}(\widehat X)
        X_n^{d-i} \right] 
  \end{align*}
  where \(\widetilde U = U \circ \pi_{n, 1}^\lambda\),
  \(\widetilde{U_i} = U_i \circ \pi_{n, 1}^\lambda\) for every \(2 \leq
  i \leq d\) and \(\widetilde{F_i} \in \formser \real {\widehat
    X}\). Let \(G \in \formser \real X\) be the series such
  that \(F(X) = X_1^d G(X)\). Since \(\widetilde U(X) = U(X_1, \dots,
  X_{n-1}, X_1X_n)\), the series \(\widetilde U\) cannot be regular in
  \(X_n\) of any order other than \(0\). Thus, the series \(G\) is
  regular in \(X_n\) of order \(d-1\). Since it suffices to show that
  we can \(*\)-monomialize the
  series \(G, X_1, \dots, X_{n-1}\) by Lemma
  \ref{monomialization-product}, the result now follows by induction
  on \(d\).

  \textbf{Case 3:} \(\lambda = 0\).
  \begin{align*}
    F \circ \pi_{n, 1}^0
    &= \widetilde U(X)X_1^d X_n^d + \sum_{i=2}^d
    \widetilde{U_i}(\widehat X) \widehat X^{\alpha_i}
      X_1^{d-i}X_n^{d-i} \\
    &= X_1^d \left[\widetilde U(X)X_n^d + \sum_{i=2}^d
      \widetilde{U_i}(\widehat X) \widehat X^{\beta_i} X_n^{d-i}\right]
  \end{align*}
  where \(\widetilde U = U \circ \pi_{n, 1}^0\),
  \(\widetilde{U_i} = U_i \circ \pi_{n, 1}^0\) for every \(2 \leq
  i \leq d\) and \(\beta_{i, j} = \alpha_{i, j}\) for \(2 \leq i \leq
  d\) and \(2 \leq j \leq n-1\) while \(\beta_{i, 1} = \alpha_{i, 1} -
  i\). Now, let
  \[G(X) = \widetilde U(X)X_n^d + \sum_{i=2}^d
    \widetilde{U_i}(\widehat X) \widehat X^{\beta_i} X_n^{d-i}.\]
  Since \(F \circ \pi_{n, 1}^0= X_1^d G\), it suffices to show that we
  can \(*\)-monomialize the series \(G, X_1, \dots, X_{n-1}\) by
  Lemma \ref{monomialization-product}. But \(\beta_l < \alpha_l\)
  so that the result follows by induction on \(\frac{\alpha_l}{l}\)
\end{proof}

\subsection{A first parametrization}
\label{par:first-parametrization}

Let \(g_0, \dots, g_n \in \germalg n\) be functions defined on
\(\polydisk n r\) and consider the \(H\)-basic set
\[A = \{x \in \polydisk n r \for g_0(x) = 0, g_1(x) > 0, \dots, g_q(x)
  > 0\}.\]
If \(\rho \colon \polydisk n {r'} \to \polydisk n r\) is an admissible
transformation then
\[B \coloneq \rho^{-1}(A) = \{x' \in \polydisk n {r'} \for g_0 \circ
  \rho(x') = 0, g_1 \circ \rho(x') > 0, \dots, g_q \circ \rho(x') >
  0\}.\]
Furthermore, if the germs at \(0\) of the functions \(g_0 \circ \rho,
\dots, g_q \circ \rho\) are all normal, then their sign in a neighborhood of
\(0\) only depends on the signs of the variables \(x_1, \dots,
x_n\). Thus, in this latter case, there is a neighborhood \(U\) of
\(0\) in \(B\) that is a finite union of sub-quadrants as defined
below. 

\begin{defn}
  A \textbf{sub-quadrant} is a set of the form \(Q = D_1 \times \dots \times
  D_n\) where each set \(D_i\) is either \(\{0\}\) or \((0, r_i)\) or
  \((-r_i, 0)\) for some \(r_i > 0\). 
\end{defn}

The discussion above combined with Theorem \ref{*-monomialization} let
us adapt \cite[Proposition 3.4]{rolin-servi-monomialization} as
follows. 

\begin{prp}
  \label{parametrization-sub-quadrants}
  Consider an \(H\)-basic set \(A\). There is a family \((Q_j,
  \rho_j)_{j \in J}\) such that
  \begin{itemize}
  \item \(Q_j \subset \real^n\) is a sub-quadrant and \(\rho_j\) is an admissible
    transformation defined on \(Q_j\) such that \(\rho_j \restriction
    Q_j\) is a diffeomorphism onto its image \(\rho_j(Q_j) \subset A\). 
  \item If \(U_j \subset \real^n\) is a neigborhood of \(0\) for each \(j \in J\) then
    there is a finite subset \(J_0 \subset J\) such that
    \[\bigcup_{j \in J_0} \rho_j(Q_j \cap U_j)\]
    is a neighborhood of \(0\) in \(A\), namely, there is a
    neighborhood \(V \subset \real^n\) of \(0\) such that \(\bigcup_{j
      \in J_0} \rho_j(Q_j \cap U_j) = V \cap A\). 
  \end{itemize}
\end{prp}

\begin{proof}
  Consider functions \(g_0, \dots, g_q \in \germalg n\) defined on
  some polydisk \(\polydisk n r\) such that
  \[A = \{x \in \polydisk n r \for g_0(x) = 0, g_1(x) > 0, \dots,
    g_q(x) > 0\}.\]
  By Theorem \ref{*-monomialization}, there is an admissible tree
  \(T\) that \(*\)-monomializes the
  functions \(g_0, \dots, g_q\). The idea of the proof is as
  follows. Consider \(\rho\) an admissible transformation induced by a
  branch of \(T\). The germs of \(g_0 \circ \rho, \dots, g_q \circ
  \rho\) at \(0\) are all normal so that, in particular, they do not
  vanish at \(0\). Thus, by
  continuity of \(g_0 \circ \rho, \dots, g_q \circ \rho\), up to
  shrinking \(Q\) we may assume that
  each function \(g_0 \circ \rho, \dots, g_q \circ \rho\) has
  constant sign on \(Q\). This implies that either \(\rho(Q) \subset
  A\) or \(\rho(Q) \cap A = \varnothing\). Now, the elementary
  transformations that are not blow-up charts have inverses for
  composition. For the
  blow-up charts, as remarked before, we need to divide by the critical
  variable in order to write the inverse. Since \(T\)
  \(*\)-monomializes \(g_0, \dots, g_q\), we can easily track the sign
  of this critical variable. Finally, the set where the critical
  variable vanishes is ``small'' so that we can
  apply the result inductively to it. We are now going to give the
  details. \\

  We prove the result by induction on
  the pairs \((n, h)\) ordered lexicographically, where \(h\) is the
  height of \(T\). If \(n \in \{0, 1\}\) or if \(h=0\), the germs \(g_0, \dots,
  g_q\) are already normal and the result follows at once. We can even
  take each \(\rho_i\) to be the identity.

  Otherwise, \(h > 0\). Assume first that the elementary
  transformations attached to the root of \(T\) are not blow-up charts and
  let \(\nu \colon \polydisk n {r'} \to \polydisk n r\) be one of
  these elementary transformations. The sub-tree
  \(T'\) of \(T\) below \(\nu\) \(*\)-monomializes the functions \(g_0
  \circ \nu, \dots, g_q \circ \nu\). Thus, the result holds for the
  \(H\)-basic set
  \[B_\nu = \{x \in \polydisk n {r'} \for g_0 \circ \nu(x) = 0, g_1
    \circ \nu(x) > 0, \dots, g_q \circ \nu(x) > 0\}. \]
  Also, \(\nu(B_\nu) \subset A\) and, since \(\nu\) can be extended
  to a diffeomorphism \(\real^n \to \real^n\), it follows that \(\nu
  \restriction B_\nu\) is a
  diffeomorphism onto its image. Furthermore, if \(U_\nu \subset \real^n\) is
  a neighborhood of \(0\) for each elementary transformation \(\nu\)
  attached to the root of \(T\), then \(\bigcup_\nu \nu(B_\nu \cap
  U_\nu)\) is a neighborhood of \(0\) in \(A\). Since there are only
  finitely many admissible transformations attached to the root of
  \(T\), the result follows at once.

  Assume next that the elementary transformations attached to the root
  of \(T\) are blow-up charts. Then, let \(w, w'\) be the two variables
  involved in this blow-up and, for each \(\nu\) attached
  to the root of \(T\), let \(w_\nu\) be its critical variable.
  Now, consider \(\nu \colon \polydisk n {r'} \to \polydisk n r\) an
  elementary transformation attached to the root of \(T\). 
  The sub-tree \(T'\) of \(T\) below \(\nu\)
  \(*\)-monomializes the functions \(g_0 \circ \nu, \dots, g_q \circ
  \nu, w_\nu\). Thus, the result holds inductively for the \(H\)-set
  \[B_\nu = \{x \in \polydisk n {r'} \for g_0 \circ \nu(x) = 0, g_1
    \circ \nu(x) > 0, \dots, g_q \circ \nu(x) > 0, w_\nu \neq 0\}.\]
  Also, \(\nu(B_\nu) \subset A\) and \(\nu \restriction B_\nu\) is a
  diffeomorphism onto its image. Assume given \(U_\nu \subset
  \real^n\) a neighborhood of \(0\) for each \(\nu\) attached to the
  root of \(T\). Then, by the compactness of \(\real \cup \{\infty\}\)
  with its usual topology
  \cite[p. 4406]{vdDries-Speissegger-generalized-power-series}, there are
  \(\nu_1, \dots, \nu_p\) some
  elementary transformations attached to the root of \(T\) such that
  \[\bigcup_{i=1}^p \nu_i(B_i \cap U_i)\]
  is a neighborhood of \(0\) in \(A \cap \{w \neq 0 \text{ or } w'
  \neq 0\}\). Furthermore, \(A \cap \{w = w' = 0\}\) can be identified
  with an \(H\)-basic set \(A' \subset \real^{n-2}\) so that the result
  now follows by induction. 
\end{proof}

\begin{remark}
  In \cite[Proposition 3.4]{rolin-servi-monomialization}, the same
  compactness as above is used in order to obtain a finite family of
  charts. However, we cannot do the same in our case. This is because
  we will use local methods that will force us to restrict the charts
  to smaller neighborhoods of \(0\). If we only kept finitely many
  charts, there would be no way to guarantee that we still cover a
  neighborhood of \(0\) in \(A\) after such restrictions. Thus, we
  will first restrict all the charts in the parametrization as much as
  we need to and then we will use compactness to obtain a finite
  family. 
\end{remark}

Nonetheless, using compactness at this stage allows us to prove that
every \(\Lambda\)-set has finitely many connected components as
follows. 

\begin{cor}
  \label{cor:lambda-finite-components}
  Let \(A \subset \real^n\) be a \(\Lambda\)-set. There are charts
  \((Q_1, \rho_1), \dots, (Q_p, \rho_p)\) such that
  \begin{itemize}
  \item For each \(1 \leq i \leq p\), \(Q_i \subset \real^n\) is a
    sub-quadrant and \(\rho_i\) is an admissible transformation
    defined on \(Q_i\).
  \item We have
    \[A = \rho_1(Q_1) \cup \dots \cup \rho_p(Q_p).\]
  \end{itemize}
  In particular, every \(\Lambda\)-set has finitely many connected
  components. 
\end{cor}

\begin{proof}
  Since \(A\) is locally an \(H\)-set around every point of \(\clos
  A\), the result follows at once from the proposition and the
  compactness of \(\clos A\). 
\end{proof}

The parametrization part of the corollary will not be reused later, it
serves only to show that \(\Lambda\)-sets have finitely many connected
components. 

\begin{defn}
  A set \(M \subset \real^n\) is called an \textbf{\(H\)-manifold} when there are a
  polydisk \(\polydisk n r\) and functions \(f_1, \dots, f_{n-d}, g_1, \dots,
  g_q \in \germalg n\) defined and \(C^1\) on \(\polydisk n r\) such
  that \(\grad f_1(x), \dots, \grad f_{n-d}(x)\) are linearly independent for
  every \(x \in \polydisk n r\) and
  \[M = \{x \in \polydisk n r \for f_1(x) = \dots = f_{n-d}(x) = 0,\
    g_1(x) > 0,\ \dots,\ g_q(x) > 0\}.\]
\end{defn}

\begin{remark}
  Assume that \(\rho\) is an admissible transformation defined on a
  sub-quadrant \(Q\). Then, the set
  \begin{align*}
    M & = \{(y, x) \in \real^n \times \real^n \for y_1 - \rho_1(x) =
        \dots = y_n - \rho_n(x) = 0, x \in Q\} \\
    &= \{(\rho(x), x) \for x \in Q\} \subset \real^{2n} 
  \end{align*}
  is an \(H\)-manifold. Furthermore, if \(\rho \restriction Q\) is a
  diffeomorphism onto its image, then \(\Pi_n \restriction M\) is also
  a diffeomorphism onto its image. 
\end{remark}

This remark provides us at once with the following restatement of
Proposition \ref{parametrization-sub-quadrants}

\begin{cor}
  \label{conj-local-parametrization}
  Let \(A \subset \real^n\) be an \(H\)-basic set. There is a family
  \((M_j)_{j \in J}\) of \(H\)-manifolds such that
  \begin{itemize}
  \item For every \(j \in J\), there is \(k_j \geq 0\) such that \(M_j
    \subset \real^{n+k_j}\). Furthermore, the projection \(\Pi_n
    \restriction M_j\) is a diffeomorphism onto its image. 
  \item If \(U_j \subset \real^{n+k_j}\) is a neighborhood of \(0\)
    for every \(j \in J\), then there is a finite subset \(J_0 \subset
    J\) such that \(\bigcup_{j \in J_0} \Pi_n(M_j \cap U_j)\) is a
    neighborhood of \(0\) in \(A\). 
  \end{itemize}
\end{cor}

\begin{proof}
  Let \((Q_j, \rho_j)_{j \in J}\) be the family provided by
  Proposition \ref{parametrization-sub-quadrants}. For each \(j \in
  J\), we define \(M_j = \{(\rho_j(x), x) \for x \in Q_j\}\). It is
  clear that the family \((M_j)_{j \in J}\) satisfies the conditions
  of the corollary.
\end{proof}

\begin{remark}
  The set \(J\) in the statement of the corollary might be infinite.
\end{remark}

The next step in the proof is to refine the statement of Corollary
\ref{conj-local-parametrization} by improving on the regularity of the
manifolds \(M_j\). The result we have in mind is Proposition
\ref{local-parametrization-constant-rank} and we will build up to it
over the course of the next two paragraphs.

\subsection{A parametrization such that \(\Pi_n\) has constant rank}

Consider \(A \subset \real^{n+k}\) an \(H\)-basic set. By applying
Corollary \ref{conj-local-parametrization} to \(A\), we
obtain a family \((M_j)_{j \in J}\) of
\(H\)-manifolds. Given \(j \in J\), we already know that \(\Pi_{n+k}
\restriction M_j\) is a diffeomorphism onto its image. Our goal in
this section will be to show that we can always choose the family
\((M_j)\) so that it is also true that the projections \(\Pi_n
\restriction M_j\) have constant rank. To be more precise, this
paragraph is devoted to proving the following proposition.

\begin{prp}
  \label{local-parametrization-first-improvement}
  Let \(A \subset \real^{n+k}\) be an \(H\)-basic set.
  Given \(0 \leq l \leq n\), there is a family
  \((M_j)_{j \in J}\) of \(H\)-manifolds such that
  \begin{itemize}
  \item For every \(j \in J\), there is \(h_j \geq 0\) such that \(M_j
    \subset \real^{n+k+h_j}\). Furthermore, the projection \(\Pi_{n+k}
    \restriction M_j\) is a diffeomorphism onto its image.
  \item If \(U_j \subset \real^{n+k+h_j}\) is a neighborhood of \(0\)
    for every \(j \in J\), then there is a finite subset \(J_0 \subset
    J\) such that \(\bigcup_{j \in J_0} \Pi_{n+k}(M_j \cap U_j)\) is a
    neighborhood of \(0\) in \(A\).
  \item For every \(j \in J\), the projection \(\Pi_n \restriction
    M_j\) has constant rank. 
  \end{itemize}
\end{prp}

In order to do so, we start
with the parametrization given in Corollary 
\ref{conj-local-parametrization} and we further
refine it by monomializing various determinants involving a basis for
\(\tangent x {M_j}\) for each \(j \in J\). Thus, the first step is to
give a basis of the tangent plane of an \(H\)-manifold in a
neighborhood of \(0\) consisting of functions in \(\germalg n\).

\begin{prp}
  \label{basis}
  Consider \(M \subset \real^n\) an \(H\)-manifold of dimension
  \(d\). There are a polydisk \(\polydisk n r\) and functions \(b_1, \dots, b_d
  \colon \polydisk n r \to \real^n\), all of whose components are in
  \(\germalg n\), and such that \(b_1(x),
  \dots, b_d(x)\) is a basis of \(\tangent x M\) for
  every \(x \in M 
  \cap \polydisk n r\).
\end{prp}

\begin{proof}
  Consider functions \(f_1, \dots, f_{n-d},
  g_1, \dots, g_q \in \germalg n\) defined and \(C^1\) on
  a polydisk \(\polydisk n r\) such that \(\grad f_1(x), \dots, \grad
  f_{n-d}(x)\) are independent for every \(x \in \polydisk n r\) and
  \[M = \{x \in \polydisk n r \for f_1(x) = \dots = f_{n-d}(x) = 0,\
    g_1(x) > 0,\ \dots,\ g_q(x) > 0\}.\]
  Now, let \(e_1, \dots, e_d \in \real^n\) be vectors such that
  \(\grad f_1(0), \dots, \grad f_{n-d}(0), e_1, \dots, e_d\) is a
  basis of \(\real^n\). By continuity of \(\det (\grad
  f_1(x), \dots, \grad f_{n-d}(x), e_1, \dots, e_d)\), we can
  assume that \(\grad f_1(x), \dots, \grad f_{n-d}(x), e_1, \dots,
  e_d\) is a basis of \(\real^n\) for every \(x \in \polydisk n
  r\) up to shrinking \(\polydisk n r\).

  For each \(x \in \polydisk n r\), we let \(a_1(x), \dots,
  a_{n-d}(x), b_1(x), \dots, b_d(x)\) be the result of applying the
  Gram-Schmidt orthonormalization process to the basis \(\grad f_1(x),
  \dots, \grad f_{n-d}(x), e_1, \dots, e_d\). Since orthonormalization
  only involves taking sums, products and dividing by non-zero
  functions, it follows that all the components of the functions
  \(a_1, \dots, a_{n-d}, b_1, \dots, b_d\) are in \(\germalg
  n\). Given \(x \in M \cap \polydisk n r\), the space \(\tangent x M\)
  is the orthogonal of the space \(V\) spanned by the vectors \(\grad
  f_1(x), \dots, \grad f_{n-d}(x)\). Now, let \(1 \leq j \leq d\) be
  an integer. Since \(V\) is also spanned by
  \(a_1(x), \dots, a_{n-d}(x)\) and since \(b_j(x)\) is orthogonal to
  each of these vectors, it follows that \(b_j(x) \in \tangent
  x M\). Thus, the vectors \(b_1(x), \dots, b_d(x)\) make up a
  linearly independent family in \(\tangent x M\). Since
  \(\dim(\tangent x M) = d\), this family must be a basis.
\end{proof}

The rank of the projection \(\Pi_n\) on an \(H\)-manifold \(M \subset
\real^{n+k}\) can be computed by looking at the sign of
various determinants. By monomializing the determinants in question,
we can cut up \(M\) into smaller pieces on which each
determinant has constant sign. In particular, we can guarantee that
\(\Pi_n\) has constant rank on the pieces with ``large'' rank. 

\begin{lem}
  \label{trivial-constant-rank}
  Let \(M \subset \real^{n+k}\) be an \(H\)-manifold with
  \(\rank x {\Pi_n \restriction M} \leq l\) for every \(x \in
  M\). Then, there are \(H\)-manifolds \(M_1, \dots, M_p\) and
  an \(H\)-basic set \(M'\) such that \(\Pi_n \restriction M_i\) has
  constant rank \(l\) for every \(1 \leq i \leq p\), \(\rank x {\Pi_n
    \restriction M} < l\) for every \(x \in M'\) and \(M_1 \cup
  \dots \cup M_p \cup M'\) is a neighborhood of \(0\) in \(M\). 
\end{lem}

\begin{proof}
  By Proposition \ref{basis}, there are a polydisk \(\polydisk {n+k}
  r\) as well as functions \(b_1, \dots, b_d \colon \polydisk {n+k} r
  \to \real^{n+k}\), all of whose components are in \(\germalg
  {n+k}\), and such that \(b_1(x), \dots, b_d(x)\) is a basis
  of \(\tangent x M\) for every \(x \in M \cap \polydisk {n+k}
  r\). Without loss of generality, we can assume that \(M \subset
  \polydisk {n+k} r\). Given
  two strictly increasing functions \(\iota \colon \{1, \dots, l\} \to
  \{1, \dots, n\}\) and \(\kappa \colon \{1, \dots, l\} \to \{1, \dots
  d\}\), we can consider the \(H\)-basic set
  \[M_{\iota, \kappa} = \{x \in M \for \det(\Pi_\iota(b_{\kappa(1)}(x)), \dots,
    \Pi_\iota(b_{\kappa(l)}(x))) \neq 0\}\]
  Notice that \(M_{\iota, \kappa}\) is the union of two \(H\)-manifolds and
  that \(\Pi_n \restriction M_{\iota, \kappa}\) has constant rank
  \(l\). Furthermore, the set
  \[M' = M \setminus \bigcup_{\iota, \kappa} M_{\iota, \kappa}\]
  is \(H\)-basic. Finally, for \(x \in M'\), we have \(\rank x {\Pi_n
  \restriction M} < l\). 
\end{proof}

In Lemma \ref{rank-cutting-parametrization} below, we refine the statement of the lemma above by
using Corollary \ref{conj-local-parametrization} in order to
parametrize the set \(M'\). We can thus replace \(M'\) with a family of
\(H\)-manifolds such that the rank of \(\Pi_n\) on each
of these manifolds is ``small''. 

\begin{lem}
  \label{rank-cutting-parametrization}
  Consider \(M \subset \real^{n+k}\) an \(H\)-manifold such that
  \(\rank x {\Pi_n \restriction M} \leq l\) for every \(x \in
  M\). Then, there is a family \((M_j)_{j \in J}\) of
  \(H\)-manifolds such that
  \begin{itemize}
  \item For every \(j \in J\), there is \(h_j \geq 0\) such that \(M_j
    \subset \real^{n+k+h_j}\). Furthermore, the projection \(\Pi_{n+k}
    \restriction M_j\) is a diffeomorphism onto its image. 
  \item If \(U_j \subset \real^{n+k_j}\) is a neighborhood of \(0\)
    for every \(j \in J\), then there is a finite subset \(J_0 \subset
    J\) such that \(\bigcup_{j \in J_0} \Pi_{n+k}(M_j \cap U_j)\) is a
    neighborhood of \(0\) in \(A\). 
  \item For every \(j \in J\), either \(\Pi_n \restriction M_j\) has
    constant rank \(l\) or \(\rank x {\Pi_n \restriction M_j} < l\)
    for every \(x \in M_j\). 
  \end{itemize}
\end{lem}

\begin{proof}
  By Lemma \ref{trivial-constant-rank}, there are \(M_1, \dots,
  M_p, M' \subset 
  \real^{n+k}\) such that the sets
  \(M_1, \dots, M_p\) are \(H\)-manifolds, \(M'\) is \(H\)-basic,
  \(M_1 \cup \dots \cup M_p \cup M'\) is a neighborhood of \(0\)
  in \(M\), \(\Pi_n \restriction M_j\) has constant rank \(l\) for \(1
  \leq j \leq p\) and \(\rank x {\Pi_n \restriction M'} < l\) for
  every \(x \in M'\). Thus, it suffices to prove the result for
  \(M'\). Now, consider a family \((M'_j)_{j \in J}\) of \(H\)-manifolds
  obtained by applying Corollary
  \ref{conj-local-parametrization} to the set \(M'\). For each \(j \in
  J\) and \(x \in M'_j\), we must have \(\rank x {\Pi_n \restriction
    M'_j} \leq \rank y {\Pi_n \restriction M'} < l\) where \(y =
  \Pi_{n+k}(x)\), whence the result. 
\end{proof}

The lemma above allows us to prove Proposition
\ref{local-parametrization-first-improvement}. The idea is that we can
continue applying the lemma to all of the pieces on which \(\Pi_n\)
does not have constant rank until the rank of \(\Pi_n\) on these pieces
goes to \(0\). The following lemma makes this idea precise. Notice
also that Proposition \ref{local-parametrization-first-improvement} is
a restatement of the case \(l=0\) in the lemma below. 

\begin{lem}
  \label{local-parametrization-constant-rank-lemma}
  Let \(A \subset \real^{n+k}\) be an \(H\)-basic set.
  Given \(0 \leq l \leq n\), there is a family
  \((M_j)_{j \in J}\) of \(H\)-manifolds such that
  \begin{itemize}
  \item For every \(j \in J\), there is \(h_j \geq 0\) such that \(M_j
    \subset \real^{n+k+h_j}\). Furthermore, the projection \(\Pi_{n+k}
    \restriction M_j\) is a diffeomorphism onto its image.
  \item If \(U_j \subset \real^{n+k+h_j}\) is a neighborhood of \(0\)
    for every \(j \in J\), then there is a finite subset \(J_0 \subset
    J\) such that \(\bigcup_{j \in J_0} \Pi_{n+k}(M_j \cap U_j)\) is a
    neighborhood of \(0\) in \(A\).
  \item For every \(j \in J\), either the projection \(\Pi_n \restriction
    M_j\) has constant rank or \(\rank x {\Pi_n \restriction M_j} \leq
    l\) for all \(x \in M_j\). 
  \end{itemize}
\end{lem}

\begin{proof}
  The result follows from Lemma
  \ref{rank-cutting-parametrization} by a decreasing induction on
  \(l\). 
\end{proof}

\subsection{Local charts}

Consider an \(H\)-manifold \(M \subset \real^{n+k}\) of dimension
\(d\) such that \(\Pi_n \restriction M\) has constant rank \(l \leq
d\). Assume also that there is a strictly increasing sequence \(\iota
\colon \{1, \dots, d\} \to \{1, \dots, n+k\}\) such that \(\Pi_\iota
\restriction M\) is an immersion. Then, \(\Pi_\iota\) is also a local
diffeomorphism so that we have a convenient way of identifying \(M\)
with \(\real^d\) locally around every point. Furthermore, assume that
\(\iota(l) \leq n\) and define \(\iota' \colon \{1, \dots, l\} \to
\{1, \dots, n\}\) by \(\iota'(i) = \iota(i)\) for every \(1 \leq i
\leq l\). Then, the following diagram commutes
\begin{center}
  \begin{tikzcd}[column sep = large, row sep = large]
    M \arrow{r}{\Pi_\iota} \arrow[']{d}{\Pi_n} &\real^d \arrow{d}{\Pi_l}
    \\
    \real^n \arrow[']{r}{\Pi_{\iota'}} &\real^l,
  \end{tikzcd}
\end{center}
showing that the pair \((\Pi_\iota, \Pi_{\iota'})\) provides us with a
local identification of \(\Pi_n \colon M \to \real^n\) with \(\Pi_l
\colon \real^d \to \real^l\). 
We may refine Corollary \ref{local-parametrization-first-improvement}
as follows to assume that such sequences exist for each \(M_j\). 

\begin{prp}
  \label{local-parametrization-constant-rank}
  Let \(A \subset \real^{n+k}\) be an \(H\)-basic set.
  There is a family \((M_j)_{j \in J}\) of \(H\)-manifolds such that
  \begin{itemize}
  \item For every \(j \in J\), there is \(h_j \geq 0\) such that \(M_j
    \subset \real^{n+k+h_j}\). Furthermore the projection \(\Pi_{n+k}
    \restriction M_j\) is a diffeomorphism onto its image.
  \item If \(U_j \subset \real^{n+k+h_j}\) is a neighborhood of \(0\)
    for every \(j \in J\), then there is a finite subset \(J_0 \subset
    J\) such that \(\bigcup_{j \in J_0} \Pi_{n+k}(M_j \cap U_j)\) is a
    neighborhood of \(0\) in \(A\).
  \item For every \(j \in J\), if \(d = \dim M_j\) then \(\Pi_n
    \restriction M_j\) has constant rank \(l\) and there is
    some sequence \(\iota \colon \{1, \dots, d\}
    \to \{1, \dots, n+k+h_j\}\) such that \(\iota\) is strictly
    increasing, \(\Pi_\iota \restriction
    M_j\) is an immersion and \(\iota(l) \leq n\).
  \end{itemize}
\end{prp}

The proposition follows at once by applying the lemma below to each
of the \(H\)-manifolds \(M_j\) in the statement of Corollary
\ref{local-parametrization-first-improvement}. The lemma itself is
simply a restatement of the discussion in
\cite{rolin-speissegger-wilkie-denjoy-carleman-classes} before Lemma
4.5. 

\begin{lem}
  \label{projection-immersion}
  Let \(M \subset \real^{n+k}\) be an \(H\)-manifold of dimension
  \(d\) such that \(\Pi_n \restriction M\) has constant rank \(l \leq
  d\). Then, there are  \(H\)-manifolds \(M_1, \dots, M_p\)
  that are also open submanifolds of \(M\) and such that
  \begin{itemize}
  \item \(M_1 \cup \dots \cup M_p\) is a neighborhood of \(0\) in
    \(M\).
  \item For each \(1 \leq i \leq p\), there exists some stricly increasing
    sequence \(\iota\colon \{1, \dots, d\} \to \{1, \dots, n+k\}\)
    such that \(\iota(l) \leq n\) and \(\Pi_\iota\restriction M_i\) is
    an immersion.
  \end{itemize}
\end{lem}

\section{Cutting fibers}
\label{cutting-fibers}

\subsection{Dimension}
\label{par:dimension}

Throughout this document, we use the word manifold to mean
\(C^1\)-submanifold of \(\real^n\) for some integer \(n \geq 1\). 
Following \cite[p. 4379]{vdDries-Speissegger-generalized-power-series}, we
say that a set \(A
\subset \real^n\) \textbf{has dimension} when it
is a countable union of manifolds. In
this case, its dimension is
\[\dim A = \max\{\dim M \for M \subset A \text{ is a manifold}\}.\]
If \(A\) is a manifold then the definition above agrees with the usual
notion of dimension. If \(A_i \subset \real^n\) has dimension for each
\(i \geq 0\) then \(A = \bigcup A_i\) also has dimension and \(\dim A
= \max\{\dim A_i \for i \geq 0\}\). Also, assume that \(M \subset
\real^n\) is a manifold and that \(f \colon M \to \real^k\) is a
\(C^1\)-map with constant rank \(r\). Then, since \(M\) has a
countable basis, the rank Theorem gives us that \(f(M)\) has dimension
and \(\dim(f(M)) = r\).

Later on, we will prove Theorem \ref{sub-sets-are-simple} by
repeatedly decreasing the dimension of the fibers. In order to do
so, we need tools to compare the dimensions
of various sets. Such results are easy consequences of the following
global parametrization for \(\Lambda\)-sets. 

\begin{lem}
  \label{global-parametrization}
  Let \(A \subset \real^{n+k}\) be a \(\Lambda\)-set. There are
  simple sub-\(\Lambda\)-sets \(M_1, \dots, M_p\) that are also
  \(H\)-manifolds and such that 
  \begin{itemize}
  \item For \(1 \leq i \leq p\), there is some \(h_i \geq 0\) such
    that \(M_i \subset \real^{n+k+h_i}\). Furthermore, the projection
    \(\Pi_{n+k} \restriction M_i\) is a diffeomorphism onto its
    image.    
  \item \(A = \Pi_{n+k}(M_1) \cup \dots \cup \Pi_{n+k}(M_p)\).
  \item For every \(i \in I\), if \(d_i = \dim M_i\) then \(\Pi_n
    \restriction M_i\) has constant rank \(l\) and there is
    some sequence \(\iota \colon \{1, \dots, d_i\}
    \to \{1, \dots, n+k+h_i\}\) such that \(\iota\) is strictly
    increasing, \(\Pi_\iota \restriction
    M_i\) is an immersion and \(\iota(l) \leq n\).
  \end{itemize}
\end{lem}

\begin{proof}
  The result follows easily from Proposition
  \ref{local-parametrization-constant-rank} and the compactness of
  \(\clos A\). 
\end{proof}

In particular, using the notation of the lemma above, we have
\(\Pi_n(A) = \Pi_n(M_1) \cup \dots \cup \Pi_n(M_p)\). Given \(1 \leq i
\leq p\), the projection \(\Pi_n\restriction M_i\) has constant rank
so that \(\Pi_n(M_i)\) has dimension. Thus, it follows that
\(\Pi_n(A)\) also has dimension. All in all, this shows that all
sub-\(\Lambda\)-sets have dimension.

Recall that, by Corollary \ref{cor:lambda-finite-components},
\(\Lambda\)-sets have finitely many connected components. The result
holds also for sub-\(\Lambda\)-sets since they are continuous images
of \(\Lambda\)-sets. Furthermore, a set \(A \subset \real^n\) has
dimension \(0\) if and only if it is discrete which yields the
following lemma. 

\begin{lem}
  \label{lem:dimension-0-finite}
  Let \(A \subset \real^n\) be a sub-\(\Lambda\)-set. Then, \(A\) has
  dimension \(0\) if and only if it is finite. 
\end{lem}

Suppose that \(M \subset \real^{n+k}\) is a manifold of dimension
\(d\) such that \(\Pi_n \restriction M\) has constant rank
\(l\). Then, we know that \(\dim M_y = d - l\) for every \(y \in
\Pi_n(M)\) and that \(\dim \Pi_n(M) = l\). In particular, \(\dim M =
\dim \Pi_n(M) + l\). The parametrization of Lemma
\ref{global-parametrization} allows us to replace \(\Lambda\)-sets
with manifolds so that we may generalize the results above to
\(\Lambda\)-sets to obtain the lemma below, the proof of which we
omit. 

\begin{lem}
  Consider \(A \subset \real^{n+k}\) a \(\Lambda\)-set such that there is an
  integer \(\mu \geq 0\) with \(\dim A_y = \mu\) for all \(y \in
  \Pi_n(A)\). Then \(\dim A = \dim (\Pi_n(A)) + \mu\). 
\end{lem}

\begin{remark}
  \label{rem:simple-sub-sets-dimension}
  In particular, notice that, if \(A \subset \real^n\) is a
  sub-\(\Lambda\)-set and if \(A' \subset \real^{n+k}\) is a
  \(\Lambda\)-set such that \(\Pi_n(A') = A\) and
  \(\Pi_n \restriction A'\) has finite fibers, then
  \(\dim A' = \dim A\).
\end{remark}

Actually, we can use the remark above to generalize the lemma to
simple sub-\(\Lambda\)-sets. 

\begin{prp}
  \label{dimension-fibers}
  Consider \(A \subset \real^{n+k}\) a simple sub-\(\Lambda\)-set and
  assume that there is an integer \(\mu \geq 0\) such that \(\dim A_y =
  \mu\) for all \(y \in \Pi_n(A)\). Then \(\dim A = \dim(\Pi_n(A)) +
  \mu\). 
\end{prp}

The proposition gives us a convenient way to argue that some sets have
small fibers. Indeed, suppose that \(A \subset B\) are two simple
sub-\(\Lambda\)-sets satisfying the hypotheses of the proposition and
such that \(\Pi_n(A) = \Pi_n(B)\) and \(\dim A < \dim B\). Then, the
fibers of \(A\) must be smaller than those of \(B\). 

\subsection{A consequence of quasianalyticity}

Suppose that \(g \colon \polydisk n r \to \real\) is a function such
that \(g \in \germalg n\) and \(g\) vanishes on an open set \(U
\subset \polydisk n r\) with \(0 \in \clos U\). Since \(U\) is open,
all the partial derivatives of \(g\) of any order vanish on \(U\) so
that, by continuity, they all vanish at \(0\) as well. Thus, by
quasianalyticity, it follows that there is an open neighborhood \(V\)
of \(0\) such that \(g\) vanishes on \(V\). The aim of this paragraph
is to show that we have a similar result when we replace \(\polydisk n
r\) with an \(H\)-manifold \(M\). To be more precise, we want to prove
the following proposition. 

\begin{prp}
  \label{vanishing}
  Consider \(\polydisk n r\) a polydisk and \(M \subset \polydisk n
  r\) an \(H\)-manifold with non-empty 
  germ at \(0\) as well as a function \(g \in \germalg n\) 
   that is defined on \(\polydisk n r\). Assume
  that the interior of the set \(\{x \in M \for g(x) = 0\}\) in \(M\) has
  non-empty germ at \(0\). Then, there is some neighborhood \(U
  \subset \real^n\) of \(0\) such that \(g\) vanishes on \(M \cap
  U\). 
\end{prp}

This result will be useful when trying to decrease dimension. Indeed,
consider \(M \subset \real^n\) an \(H\)-manifold of dimension \(d\) and
let \(g \in \germalg n\). The zero set of \(g \restriction M\) defines
a germ of subsets of \(M\) which we write \(Z_M(g)\). If \(\dim Z_M(g) =
d\) then the interior of \(Z_M(g)\) in \(M\) has non-empty germ at \(0\)
so that \(g\) vanishes in a neighborhood of \(0\) in \(M\). Thus,
either \(Z_M(g)\) and \(M\) have the same germ at \(0\) or \(\dim Z_M(g) <
d\).

\begin{proof}
  Up to shrinking \(\polydisk n r\), there are functions \(f_1, \dots,
  f_{n-d} \in \germalg n\) defined and \(C^1\) on \(\polydisk n r\)
  such that \(M\) is an open submanifold of
  \[M' = \{x \in \polydisk n r \for f_1(x) = \dots = f_{n-d}(x) =
    0\}\]
  and \(\grad f_1(x), \dots, \grad f_{n-d}(x)\) are linearly
  independent for every \(x \in \polydisk n r\). Without loss of
  generality, we may assume that \(M = M'\).

  The map \(f = (f_1, \dots, f_{n-d}) \colon \polydisk n r \to
  \real^{n-d}\) is a submersion. Thus, by stability of the algebras
  generated by \(H\) under implicit
  functions, there exist a polydisk \(\polydisk d s\), a neighborhood
  \(U \subset \real^n\) of \(0\) and a diffeomorphism \(\Phi \colon
  \polydisk d s \to M \cap U\) all of whose components are in
  \(\germalg d\). We may now apply the result to the function \(g
  \circ \Phi\) because it is defined on a polydisk, which allows us to
  conclude.
\end{proof}

\subsection{The Local Fiber Cutting Lemma}

At the moment, we know from Lemma \ref{global-parametrization} that we can
parametrize sub-\(\Lambda\)-sets \(A \subset \real^n\) as
\[A = \Pi_n(M_1) \cup \dots \cup \Pi_n(M_p)\]
where each \(M_i\) is an \(H\)-manifold such that \(\Pi_n \restriction
M_i\) has constant rank.
However, there is no bound on the dimension of the fibers.
In order to show that the hypotheses of Gabrielov's
Theorem of the Complement are satisfied, we need to have such a
parametrization where, for each \(1 \leq i \leq p\), there is a
strictly increasing sequence \(\iota \colon \{1, \dots, d_i\} \to \{1,
\dots, n\}\) such that \(\Pi_{\iota} \restriction M_i\) is an
immersion, where \(d_i = \dim(M_i)\). This implies in particular that
\(\Pi_n \restriction M_i\) is an immersion so that its fibers must
have dimension \(0\) whence they are finite by Lemma
\ref{lem:dimension-0-finite}.
We will show over the next two paragraphs that such parametrizations
exist. This paragraph is dedicated to proving the Local Fiber Cutting
Lemma, which is the main tool we use in the proof, while Paragraph
\ref{par:global-fiber-cutting} is concerned with showing how this
lemma can be used to conclude.

In order to prove the Local Fiber Cutting Lemma, we first need two
results. If \(M \subset \real^{n+k}\) is an \(H\)-manifold such that
\(\Pi_n \restriction M\) has constant rank, then each fiber \(M_x\),
with \(x \in \Pi_n(M)\), is a manifold. As in \ref{basis}, we construct a
basis for this manifold. The second result will allow us to derive a
contradiction in the proof of the Local Fiber Cutting Lemma. 

\begin{lem}
  \label{basis-fibers}
  Let \(M \subset \real^{n+k}\) be an \(H\)-manifold of dimension \(d\)
  and assume that \(\Pi_n \restriction M\) has constant rank \(l \leq
  d\). Then there are a polydisk \(\polydisk {n+k} r\) and functions 
  \(b_1, \dots, b_{d-l} \colon \polydisk {n+k} r \to \real^{n+k}\),
  all of whose components are in \(\germalg {n+k}\), and such that 
  for all \(z \in M \cap \polydisk {n+k} r\), \(b_1(z), \dots,
  b_{d-pl}(z)\) is a basis of \(\tangent z
  {\{x\} \times M_x}\) where \(x = \Pi_n(z)\). 
\end{lem}

\begin{proof}
  There are a polydisk \(\polydisk {n+k} r\) and functions \(a_1,
  \dots, a_d \colon \polydisk {n+k} r \to \real^{n+k}\), all of whose
  components are in \(\germalg {n+k}\), and such that \(a_1(z),
  \dots, a_d(z)\) is a basis of \(\tangent z M\) for all \(z \in
  M \cap \polydisk {n+k} r\). Up to reordering these functions, we
  may assume that
  \(\Pi_n(a_1(0)), \dots, \Pi_n(a_l(0))\) is a basis of
  \(\Pi_n(\tangent 0 M)\). By the hypothesis on constant rank, up to
  shrinking \(\polydisk {n+k} r\), we may also assume that
  \(\Pi_n(a_1(z)), \dots,
  \Pi_n(a_l(z))\) is a basis of \(\Pi_n(\tangent z M)\) for all \(z
  \in M \cap \polydisk {n+k} r\).

  Then, define functions \(\widetilde{a_1}, \dots, \widetilde{a_l}
  \colon \polydisk {n+k} r \to \real^{n+k}\)
  inductively as
  \begin{align*}
    \widetilde{a_1}(z) &= \frac{a_1(z)}{\norm{a_1(z)}}\\[10pt]
    \widetilde{a_i}(z) &= \frac{a_i(z) - \sum_{j < i} \langle
                         \Pi_na_i(z), \Pi_n \widetilde{a_j}(z) \rangle
                         \widetilde{a_j}(z)}{\norm{\Pi_na_i(z) -
                         \sum_{j < i} \langle \Pi_n a_i(z), \Pi_n
                         \widetilde{a_j}(z) \rangle \Pi_n
                         \widetilde{a_j}(z)}},
  \end{align*}
  where \(\langle \cdot, \cdot \rangle\) is the scalar product. 
  It is easy to see that the components of \(\widetilde{a_1}, \dots,
  \widetilde{a_l}\) are in \(\germalg {n+k}\) and that
  \(\Pi_n\widetilde{a_1}(z), \dots, \Pi_n\widetilde{a_l}(z)\) is an
  orthonormal basis of \(\Pi_n(\tangent z M)\) for every \(z \in M
  \cap \polydisk {n+k} r\).

  Finally, we define functions \(b_1,
  \dots, b_{d-l}\colon \polydisk {n+k} r \to \real^{n+k}\) as
    \[b_i(z) = a_{l+i}(z) - \sum_{j=1}^l \langle \Pi_n a_{l+i}(z),
    \Pi_n\widetilde{a_j}(z) \rangle \widetilde{a_j}(z).\]
  The components of the functions \(b_1, \dots, b_{d-l}\) are in
  \(\germalg {n+k}\) and the family
  \(\widetilde{a_1}(z), \dots, \widetilde{a_l}(z), b_1(z), \dots,
  b_{d-l}(z)\) is a basis of 
  \(\tangent z M\) for all \(z \in M \cap \polydisk {n+k} r\). Furthermore,
  \(\Pi_n(b_i(z)) = 0\) for
  \(z \in M\) so that the vectors \(b_1(x), \dots, b_{d-l}(x)\) make
  up a basis for \(\tangent z {\{x\} \times M_x}\). 
\end{proof}

\begin{remark}
  Notice that the formulas giving the functions \(\widetilde{a_1},
  \dots, \widetilde{a_n}\) are reminiscent of the Gram-Schmidt
  orthonormalization algorithm. However, we have included the formulas
  because there is a subtlety, namely, we do not necessarily want
  \(\widetilde{a_1}(z), \dots, \widetilde{a_n}(z)\) to be
  orthonormal. Rather, we wish to show that their projections on the
  first \(n\) coordinates are orthonormal. 
\end{remark}

The following lemma will be crucial in order to obtain a contradiction
in the proof of the Local Fiber Cutting Lemma below. It is proven in
\cite{rolin-speissegger-wilkie-denjoy-carleman-classes}, in the
paragraph before their Lemma 4.5.

\begin{lem}
  \label{frontier}
  Let \(M \subset \real^{n+k}\) be an \(H\)-manifold of dimension
  \(d\) such that \(\Pi_n \restriction M\) has constant rank \(l <
  d\). Assume also that
  there is a strictly increasing sequence \(\iota \colon \{1, \dots,
  d\} \to \{1, \dots, n+k\}\) such that \(\Pi_\iota \restriction M\)
  is an immersion and \(\iota(l) \leq n\). Consider \(y \in \Pi_n(M)\)
  and let \(C\) be a 
  connected component of the fiber \(M_y\). Then, \(\frontier C \neq
  \varnothing\). 
\end{lem}

\begin{prp}[Local Fiber Cutting Lemma]
  \label{local-fiber-cutting}
  Let \(M \subset \real^{n+k}\) be an \(H\)-manifold of dimension
  \(d\) such that \(\Pi_n \restriction M\) has constant rank \(l <
  d\). Assume also that there is a strictly increasing sequence \(\iota
  \colon \{1, \dots, d\} \to \{1, \dots, n+k\}\) with \(\iota(l) \leq
  n\) and such that \(\Pi_\iota \restriction M\) is an immersion.
  Then, there are a polydisk \(\polydisk n r\) and a simple
  sub-\(\Lambda\)-set \(A
  \subset M\) such that \(\dim A < \dim M\) and \(\Pi_n(M \cap
  \polydisk {n+k} r) = \Pi_n(A)\). 
\end{prp}

\begin{proof}
  For \(z \in \real^{n+k}\), we will write throughout the proof \(z =
  (x,  y)\) where \(x = (x_1, \dots, x_n)\) and \(y = (y_1, \dots,
  y_k)\). Thus, in particular, \(x = \Pi_n(z)\). We assume throughout
  that \(M\) has non-empty germ at \(0\) since there is nothing to
  prove otherwise. 
  There are a polydisk \(\polydisk {n+k} r \subset \real^{n+k}\), a
  natural number \(q \in \nat\) and functions \(f_1, \dots,
  f_{n+k-d}, g_1, \dots, g_q \in \germalg
  {n+k}\) that are defined and \(C^1\) on \(\polydisk
  {n+k} r\), such that \(\grad f_1(x), \dots, \grad f_{n+k-d}(x)\) are
  linearly independent for every \(x \in \polydisk {n+k} r\) and
  \[M = \{x \in \polydisk {n+k} r \for f_1(x) = \dots = f_{n+k-d}(x) =
    0, g_1(x) > 0, \dots, g_q(x) > 0\}.\]
  By Lemma \ref{basis-fibers}, up to shrinking \(\polydisk {n+k} r\), we may
  assume that there are functions \(b_1, \dots, b_{d-l} \colon
  \polydisk {n+k} r \to \real^{n+k}\), all of whose components are in
  \(\germalg {n+k}\) and such that \(b_1(z), \dots, b_{d-l}(z)\) make
  up a basis of \(\tangent z {\{x\} \times M_x}\) for each \(z \in
  \real^{n+k}\). We are going to prove the result in two steps. We
  will begin by proving the result under an additional assumption
  about \(M\) and we will then reduce the general case to the special
  one. \\

  \textit{Special case:} For now, we assume that, for each \(1 \leq i
  \leq k\), there is some \(1 \leq j \leq n\) such that \(\abs{y_i} <
  \abs{x_j}\) whenever \(z = (x,y) \in M\).
  This hypothesis means that \(M\) is contained in a ``butterfly
  shape''. We have represented the case \(n = k =1\) in the picture on
  the left in figure \ref{fig:fiber-cutting}, where \(M\) would be
  contained in the orange area.

  \begin{figure}
    \begin{tikzpicture}[scale = 0.6]
      \draw[step=1cm,gray,very thin,opacity=0.5] (-4.9,-4.9) grid
      (4.9,4.9);
      \draw[thick,->] (-4.5,0) -- (4.5,0) node[anchor =
      north east]{\(x\)};
      \draw[thick,->] (0,-4.5) -- (0,4.5)
      node[anchor = north east]{\(y\)}; \draw (-4, -4) -- (4, 4);
      \draw (-4, 4) -- (4, -4);
      \fill [orange!60!white, opacity = 0.3] (-4, 4) -- (-4, -4) --
      (0, 0) -- cycle; \fill [orange!60!white, opacity = 0.3] (4, 4)
      -- (4, -4) -- (0, 0) -- cycle;
    \end{tikzpicture}
    \hspace{20pt}
    \begin{tikzpicture}[scale=0.6]
      \draw[step=1cm,gray,very thin,opacity=0.5] (-4.9,-4.9) grid
      (4.9,4.9);
      \draw[thick,->] (-4.5,0) -- (4.5,0) node[anchor = north east]{\(x\)};
      \draw[thick,->] (0,-4.5) -- (0,4.5) node[anchor = north east]{\(y\)};
      \draw (-4, -4) -- (4, 4);
      \draw (-4, 4) -- (4,-4);
      \draw (-3, -2) rectangle (3, 2) node[anchor=north west]
      {\(\polydisk {2(n+k)} {s'}\)};
      \fill [green!60!white, opacity = 0.3] (-2, -2) rectangle (2, 2);
    \end{tikzpicture}
    \caption{}
    \label{fig:fiber-cutting}
  \end{figure}

  We define a function \(\varphi \in \germalg {n+k}\) defined and
  \(C^1\) on \(\polydisk {n+k} r\) by
  \[\varphi(z) = g_1(z) \dots g_q(z) (r_1^2 - z_1^2) \dots (r_{n+k}^2
    - z_{n+k}^2).\] 
  If \(z \in M\) then \(\varphi(z) > 0\) but 
  \(\varphi(z) = 0\) whenever \(z \in \frontier M = \clos M \setminus
  M\). Now, let \(x \in \Pi_n(M)\).
  Since \(\clos {M_x}\) is compact and non-empty, there is \(y_0 \in
  \clos {M_x}\) such that \(\varphi(x, y_0)\) is maximal. From \(y_0
  \in \clos {M_x}\), we deduce that \((x, y_0) \in \clos M\). 
  Also, for every \(y \in M_x\), we have \(\varphi(x, y) > 0\) so that
  \(\varphi(x, y_0) > 0\) by maximality. Thus, \((x, y_0) \not \in
  \frontier M\) so that \((x, y_0) \in M\) whence \(y_0 \in
  M_x\). Furthermore, \(\varphi \restriction \{x\} \times M_x\) is
  critical at \(z\). Thus, if we let
  \[A = \{z \in M \for \varphi \restriction \{x\} \times M_x \text{ is
      critical at } z\},\]
  then we obtain \(\Pi_n(M) = \Pi_n(A)\).
  We also have
  \[A = \{z \in M \for \langle \grad \varphi(z), b_1(z) \rangle =
    \dots = \langle \grad \varphi(z), b_{d-l}(z) \rangle = 0\}\]
  so that \(A\) is a \(H\)-basic set. \\

  A polydisk \(\polydisk {n+k} {r'} \subset \polydisk {n+k} r\), with
  \(r' = (s', t')\), is said to be \textbf{compatible} with \(M\)
  whenever we have \((M \cap \polydisk {n+k} {r'})_x = M_x\) for every
  \(x \in \real^n\) such that \((M \cap \polydisk {n+k} {r'})_x \neq
  \varnothing\). Such polydisks form a fundamental system of
  neighborhoods of \(0\), as shown in the picture on the right in
  figure \ref{fig:fiber-cutting}, where the green area is a polydisk
  compatible with \(M\).
  Furthermore, let \(\polydisk {n+k} {r'}\) be a polydisk compatible
  with \(M\) and \(x \in \Pi_n(M \cap \polydisk {n+k} {r'})\). Take
  \(y\) such that \(z = (x, y) \in A\). Since \(A 
  \subset M\), we have \(y \in M_x\) so that \(z \in \polydisk
  {n+k}{r'}\) by the compatibility assumption. Thus,
  \[\Pi_n(M \cap \polydisk {n+k} {r'}) = \Pi_n(A \cap \polydisk {n+k}
    {r'}).\]
  In view of the equality above, the fact that \(A\) is \(H\)-basic
  and Theorem \ref{basic-simple-sub-lambda-set}, it suffices to prove
  that there exists a polydisk \(\polydisk {n+k} {r'}\) which is
  compatible with \(M\) and such that \(\dim(A \cap
  \polydisk {n+k} {r'}) < \dim(M)\). \\

  Suppose for a contradiction that this is not the case and let \(B\)
  be the interior of \(A\) in \(M\). Then, for every polydisk
  \(\polydisk {n+k} {r'} \subset \polydisk {n+k} r\) which is
  compatible with \(M\), we have \(\dim(A \cap
  \polydisk {n+k} {r'}) = \dim(M)\) by assumption so that \(A \cap
  \polydisk {n+k} {r'}\) has non-empty interior in \(M\). In
  particular, \(B \cap \polydisk {n+k} {r'} \neq \varnothing\). Since
  such polydisks form a fundamental system of neighborhoods of \(0\),
  it follows that \(0 \in \clos B\). By Proposition \ref{vanishing},
  there exists a polydisk \(\polydisk {n+k} {r'} \subset \polydisk
  {n+k} r\) which is compatible with \(M\) and such that
  \(M \cap \polydisk {n+k} {r'} \subset A\).

  Let \(r' = (s', t')\) and consider \(x \in \Pi_n(M) \cap \polydisk n
  {s'}\). By compatibility of \(\polydisk {n+k}{r'}\), it follows that
  \(M_x \subset \polydisk {n+k}{r'}\). Thus, \(A_x = M_x\) whence
  \(\varphi_x = \varphi(x, \cdot)\) is constant along any connected
  component of \(M_x\). 
  Let \(C\) be such a connected component. Then, \(\varphi_x\)
  must also be constant on \(\clos C\) and, since \(\varphi_x\) is
  positive on \(C\) and vanishes on \(\frontier C\), it
  follows that \(\frontier C = \varnothing\) which is a contradiction
  by Lemma \ref{frontier}. \\

  \textit{General case:} We now return to the general case. We are
  looking for some polydisk \(\polydisk {n+k} {r'}\) such that there
  exists a simple sub-\(\Lambda\)-set \(A \subset M \cap \polydisk
  {n+k} {r'}\) with \(\dim A < \dim M\) and \(\Pi_n(M \cap \polydisk
  {n+k} {r'}) = \Pi_n(A)\). We may view \(r'\) as some sort of
  parameter. The idea of the proof is to replace \(M\) with a
  \(H\)-manifold in \(\real^{2(n + k)}\), where we have turned the
  parameter \(r'\) into a tuple of variables. The precise definition
  of \(\widetilde M\) is as follows:
  \[\widetilde M = \{(r', z) \for 0 < r_i' <  r_i \text{ for all }
    1 \leq i \leq n+k \text{ and } z \in M \cap \polydisk {n+k}
    {r'}\}.\]
  In particular, for every \(r' \in \real^{n+k}\) such
  that \(0 < r'_i < r_i\) for each \(1 \leq i \leq n+k\), we have
  \[\widetilde M_{r'} = M \cap \polydisk {n+k} {r'}.\]
  Thus, \(\widetilde M\) can be thought of as a parametrized family of
  restrictions of \(M\). 
  
  Notice also that \(\widetilde M\) is a \(H\)-manifold of dimension
  \(n+k+d\) such that \(\Pi_{2n+k} \restriction \widetilde M\) has constant
  rank \(n+k+l < n+k+d\) and that \(\widetilde M\) is contained in a
  butterfly shape. Indeed, \(\abs {y_i} < t'_i\) for every \(1
  \leq i \leq k\) and every \((r', z) \in M\), where \(z = (x,
  y)\) and \(r' = (s', t')\). By the special case of the result, there
  are a simple
  sub-\(\Lambda\)-set \(\widetilde A \subset \widetilde M\) and a
  polydisk \(\polydisk {2(n+k)} R\) such that \(\Pi_{2n+k}(\widetilde
  M \cap \polydisk {2(n+k)} R) = \Pi_{2n+k}(\widetilde A)\) and
  \(\dim(\widetilde A) < \dim(\widetilde M)\).

  We now want to use Proposition \ref{dimension-fibers} to show that
  we can find \(r'\) arbitrarily close to \(0\) such that
  \(\dim(\widetilde A_{r'}) < \dim(M)\). 
  Thus, suppose for a contradiction that there is \(r''\) such that
  \(\dim(\widetilde
  A_{r'}) = \dim(M)\) whenever \(0 < r'_i < r''_i\) for each \(1 \leq
  i \leq n+k\). It follows that all the fibers of \(\widetilde A \cap
  (\polydisk {n+k} {r''} \times \real^{n+k})\) have the same
  dimension. Furthermore,
  \begin{align*}
    \dim(\Pi_{n+k}(\widetilde A \cap (\polydisk {n+k}{r''} \times
    \real^{n+k}))) 
    &= \dim(\Pi_{n+k}(\widetilde A) \cap \polydisk {n+k}{r''}) \\
    &= \dim(\Pi_{n+k}(\widetilde M) \cap \polydisk {n+k} {r''}) \\
    &= n+k.
  \end{align*}
  Thus, by Proposition \ref{dimension-fibers}, we obtain
  \(\dim(\widetilde A \cap (\polydisk {n+k} {r''} \times \real^{n+k}))
  = \dim(M) + n+k = \dim(\widetilde M)\) which is a contradiction. 

  Thus, it is now sufficient to show that there is \(r''\) as above
  such that \(\Pi_n(\widetilde A_{r'}) = \Pi_n(M \cap \polydisk {n+k}
  {r'})\) for every \(r' \in \real^{n+k}\) such that \(0 < r'_i <
  r''_i\) whenever \(1 \leq i \leq n+k\). Up to shrinking
  \(\polydisk {2(n+k)} R\), we may assume
  that \(R\) has the form \((r'', r'')\) for some polyradius
  \(r''\). Then, given \(r'\) such that \(0 < r'_i < r''_i\) for every
  \(1 \leq i \leq n+k\), we have \(\{r'\} \times \widetilde M_{r'}
  \subset \polydisk {2(n+k)} R\) whence \(\Pi_n(\widetilde M_{r'}) =
  \Pi_n(\widetilde A_{r'})\). Furthermore, by definition of \(\widetilde
  M\), we have \(\widetilde M_{r'} = M \cap \polydisk {n+k} {r'}\)
  which concludes. 
\end{proof}

\begin{remark}
  We reuse the notation of the proof. If we assume
  \(\mathcal{A}\)-analyticity \cite[Definition
  1.10]{rolin-servi-monomialization}, then we can show that \(B\) is
  both open and closed. Indeed, consider \(a \in \clos B\). By
  \(\mathcal{A}\)-analyticity, the translation of \(A\) at \(a\) is
  still a set defined by equations so that we may apply
  Lemma \ref{vanishing} to deduce that \(a \in B\) . Thus, \(B\)
  contains every connected component that it intersects which yields a
  contradiction whenever \(B \neq \varnothing\). Notice that this
  argument does not depend on \(M\) being contained in a butterfly
  shape and that it proves the stronger result that \(\dim A < \dim
  M\).

  However, in the absence of \(\mathcal{A}\)-analyticity, we can only
  show that, when \(0 \in \clos B\), we must also have \(0 \in
  B\). Thus, we only get a contradiction when the germ of \(B\) at
  \(0\) is non-empty whence we can only prove that the dimension of
  \(A\) is small \textit{in a neighborhood of \(0\)}. Furthermore,
  this contradiction depends on the fact that every neighborhood of
  \(0\) contains a connected component of a fiber of \(M\). This does
  not hold in the general case (see Appendix \ref{appendix}) but it is
  true in the special case that
  \(M\) is contained in a butterfly shape (see the picture on the
  right in figure \ref{fig:fiber-cutting}).

  Actually, the proof we give for the special case still applies when
  the butterfly shape is ``distorted'', namely when the lines bounding
  the butterfly are not straight. More precisely, suppose that there
  are \(C^1\) functions \(h_1, \dots, h_k \in \germalg n\) defined on
  \(\polydisk n s\), where \(r = (s, t)\), such
  that \(h_1(0) = \dots = h_k(0) = 0\) and \(\abs{y_i} <
  h_i(x)\) for every \(z = (x, y) \in M\) and \(1 \leq i \leq
  k\). Then, it is still true that the polydisks compatible with \(M\)
  form a fundamental system of neighborhoods of \(0\) so that the
  proof given for the special case above applies in this case as
  well. 
\end{remark}

\subsection{The Global Fiber Cutting Lemma}
\label{par:global-fiber-cutting}

In this paragraph, we start by giving a global version of the Fiber
Cutting Lemma. We use it inductively to prove that every
sub-\(\Lambda\)-set is simple. Finally, using Gabrielov's Theorem of
the Complement \cite[Theorem
2.7]{vdDries-Speissegger-generalized-power-series}, we show that the
\(\struct\) is o-minimal and model complete. 

\begin{prp}
  Consider \(B \subset \real^{n+k}\) an \(H\)-basic
  set such that \(\dim(\Pi_n(B)) < \dim(B)\). Then, there are
  \(\Lambda\)-sets \(A_1, \dots, A_p\) such that
  \begin{itemize}
  \item For every \(1 \leq i \leq p\), there is some integer \(h_i
    \geq 0\) such that \(A_i \subset \real^{n+k+h_i}\), \(\dim(A_i)
    < \dim(B)\), \(\Pi_{n+k}(A_i) \subset B\) and \(\Pi_{n+k}
    \restriction A_i\) has finite fibers.
  \item There is a neighborhood \(U \subset \real^{n+k}\) of \(0\)
    such that
    \[\Pi_n(B \cap U) = \Pi_n(A_1) \cup \dots \cup \Pi_n(A_p).\]
  \end{itemize}
\end{prp}

\begin{proof}
  The proof proceeds by first parametrizing \(H\)-basic sets
  using Proposition \ref{local-parametrization-constant-rank} and then
  applying the Local
  Fiber Cutting Lemma to each of these charts. The following diagram
  illustrates the situation.
  
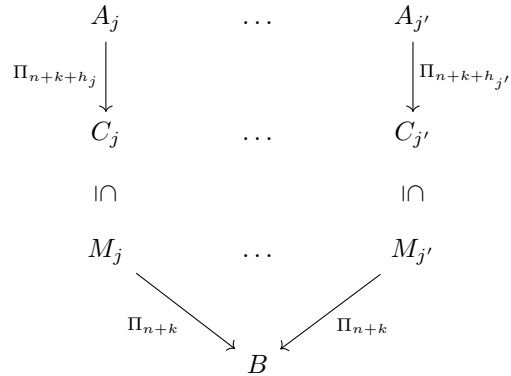
\begin{center}
    \begin{tikzcd}[column sep = large, row sep = large]
      A_{j} \arrow[']{d}{\Pi_{n+k+h_j}} &\dots &A_{j'} \arrow{d}{\Pi_{n+k+h_{j'}}} \\
      C_{j} \arrow[phantom, sloped]{d}{\subset} &\dots &C_{j'}
      \arrow[phantom, sloped]{d}{\subset}\\
      M_{j} \arrow[']{dr}{\Pi_{n+k}} &\dots &M_{j'} \arrow{dl}{\Pi_{n+k}}\\
      &B
    \end{tikzcd}
  \end{center}
  
  Consider \((M_j \subset \real^{n+k+h_j})_{j \in J}\) a parametrization of
  \(B\) obtained by Proposition
  \ref{local-parametrization-constant-rank} and fix \(j \in J\). If
  \(\dim(M_j) < \dim(B)\) then we let \(C_j = M_j\) and \(U_j =
  \real^{n+k+h_j}\). Notice in particular that \(\Pi_n(C_j) =
  \Pi_n(M_j \cap U_j)\). Now, assume
  that \(\dim(M_j) = \dim(B)\). Then, \(\Pi_n \restriction M_j\) has
  constant rank \(l = \dim(\Pi_n(M_j)) \leq \dim(\Pi_n(B)) < \dim(B) =
  \dim(M_j)\). Thus, we can apply the Local Fiber Cutting Lemma to
  find \(C_j \subset M_j\) a simple sub-\(\Lambda\)-set and \(U_j
  \subset \real^{n+k+h_j}\) a neighborhood of \(0\) such that
  \(\dim(C_j) < \dim(M_j)\) and \(\Pi_n(C_j) = \Pi_n(M_j \cap U_j)\). 
 
  Now, consider \(J_0 \subset J\) a finite subset such that
  \(V \coloneq \bigcup_{j \in J_0} \Pi_{n+k}(M_j \cap U_j)\) is a
  neighborhood of \(0\) in \(B\). We then have
  \[\Pi_n(V) = \bigcup_{j \in J_0} \Pi_n(M_j \cap U_j) =
    \bigcup_{j \in J_0} \Pi_n(C_j).\]
  For \(j \in J_0\), \(C_j\) is a simple sub-\(\Lambda\)-set so that
  there is
  a \(\Lambda\)-set \(A_j\) such that
  \(\Pi_{n+k+h_j}(A_j) = C_j\) and \(\dim(A_j) = \dim(C_j)\). In
  particular, \(\dim(A_j) < \dim(B)\) for \(j \in J_0\) and
  \[\Pi_n(V) = \bigcup_{j \in J_0} \Pi_n(A_j)\]
  whence the result. 
\end{proof}

\begin{cor}[Global Fiber Cutting]
  \label{global-fiber-cutting}
  Let \(B \subset \real^{n+k}\) be a \(\Lambda\)-set such that
  \(\dim(\Pi_n(B)) < \dim(B)\). Then, there are \(\Lambda\)-sets
  \(A_1, \dots, A_p\) such that
  \begin{itemize}
  \item For every \(1 \leq i \leq p\), there is some integer \(h_i
    \geq 0\) such that \(A_i \subset \real^{n+k+h_i}\), \(\dim(A_i)
    < \dim(B)\), \(\Pi_{n+k}(A_i) \subset B\) and \(\Pi_{n+k}
    \restriction A_i\) has finite fibers.
  \item We have
    \[\Pi_n(B) = \Pi_n(A_1) \cup \dots \cup \Pi_n(A_p).\]
  \end{itemize}
\end{cor}

\begin{proof}
  The result follows from the proposition above and the compactness of
  \(\clos B\). 
\end{proof}

\begin{remark}
  \label{rem:global-fiber-cutting-fibers-smaller}
  We require that \(\Pi_{n+k} \restriction A_i\) has finite fibers and
  that \(\Pi_{n+k}(A_i) \subset B\) so that, for every \(y \in
  \Pi_n(A_i)\), we have \(\dim((A_i)_y) \leq \dim(B_y)\). Indeed,
  consider \(y \in \Pi_n(A_i)\). We have \(\Pi_{n+k}((A_i)_y) \subset
  B_y\) and \(\Pi_{n+k} \restriction (A_i)_y\) has finite fibers so
  that, by Proposition \ref{dimension-fibers}, \(\dim((A_i)_y) =
  \dim(\Pi_{n+k}((A_i)_y)) \leq \dim(B_y)\). 
\end{remark}

We want to work with \(\Lambda\)-sets that respect the hypotheses of
Proposition \ref{dimension-fibers} in order to have better control
over their fibers. This
is the purpose of the following lemma, which is an immediate
consequence of Lemma \ref{global-parametrization}.

\begin{lem}
  \label{lambda-sets-constant-rank}
  Let \(B \subset \real^{n+k}\) be a \(\Lambda\)-set. Then, there are
  \(\Lambda\)-sets \(A_1, \dots, A_p\) such that
  \begin{itemize}
  \item For every \(1 \leq i \leq p\), there is some integer \(h_i
    \geq 0\) such that \(A_i \subset \real^{n+k+h_i}\) and
    \(\dim((A_i)_y)\) does not depend on \(y \in \Pi_n(A_i)\). 
  \item For each \(1 \leq i \leq p\), \(\Pi_{n+k} \restriction A_i\)
    has finite fibers and 
    \[B = \Pi_{n+k}(A_1) \cup \dots \cup \Pi_{n+k}(A_p).\]
  \end{itemize}  
\end{lem}

\begin{thm}
  \label{sub-sets-are-simple}
  Let \(A \subset \real^n\) be a sub-\(\Lambda\)-set. Then, \(A\) is
  also a simple sub-\(\Lambda\)-set. 
\end{thm}

\begin{proof}
  The following diagram illustrates the proof below. 
  
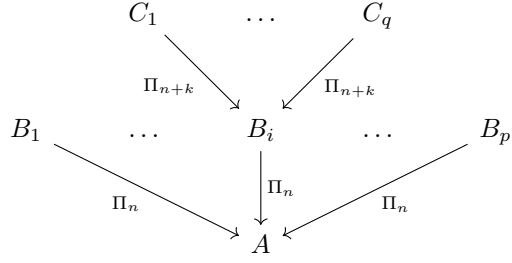
\begin{center}
    \begin{tikzcd}[row sep = large]
      &C_1 \arrow[']{dr}{\Pi_{n+k}} &\dots &C_q \arrow{dl}{\Pi_{n+k}} \\
      B_1 \arrow[']{drr}{\Pi_n} &\dots &B_i \arrow{d}{\Pi_n} &\dots
      &B_p \arrow{dll}{\Pi_n} \\
      &&A
    \end{tikzcd}
  \end{center}
  
  It suffices to show that there are \(\Lambda\)-sets \(B_1, \dots,
  B_p\) such that \(\Pi_n \restriction B_i\) has finite fibers and \(A
  = \Pi_n(B_1) \cup \dots \cup \Pi_n(B_p)\). Thus, consider \(B_1,
  \dots, B_p\) some \(\Lambda\)-sets such that \(A = \Pi_n(B_1) \cup
  \dots \cup \Pi_n(B_p)\). By applying Lemma
  \ref{lambda-sets-constant-rank}, we may assume
  that, for every \(1 \leq i \leq p\), the dimension \(\mu_i \coloneq
  \dim((B_i)_y)\) does not depend on \(y \in \Pi_n(B_i)\). Define also
  \(d_i = \dim(\Pi_n(B_i))\) for each \(1 \leq i \leq p\). Finally,
  consider \(\mu = \max(\mu_1, \dots, \mu_p)\) and \(d = \max\{d_i \for \mu_i
  = \mu\}\) and assume that \(B_1, \dots, B_p\) have been chosen in such
  a way that the pair \((\mu, d)\) is minimal for the lexicographic
  order. To prove the result, we only need to show that \(\mu =
  0\) by Lemma \ref{lem:dimension-0-finite}. Thus, assume
  by contradiction that \(\mu > 0\).

  Let \(1 \leq i \leq p\) such that \(\mu_i = \mu\) and \(d_i =
  d\). For the rest of the proof, we will write \(B = B_i\). By Proposition
  \ref{dimension-fibers}, we have \(\dim(B) = d + \mu > d\) so
  that, by Corollary \ref{global-fiber-cutting}, there are \(C_1,
  \dots, C_{q}\) some \(\Lambda\)-sets such that
  \[\Pi_n(B) = \Pi_n(C_1) \cup \dots \cup \Pi_n(C_{q})\]
  and \(\dim (C_j) < \dim(B)\) for \(1 \leq j \leq q\). By Lemma
  \ref{lambda-sets-constant-rank}, we can assume without loss of
  generality that, for every \(1 \leq j
  \leq q\), the dimension \(\mu_j \coloneq \dim((C_j)_y)\) does not
  depend on \(y \in \Pi_n(C_j)\). Define
  also \(d_j = \dim(\Pi_n(C_j))\) for \(1 \leq j \leq q\). Then,
  fix \(1 \leq j \leq q\). By Remark
  \ref{rem:global-fiber-cutting-fibers-smaller}, we have \(\mu_j \leq
  \mu\).
  Assume that \(\mu_j = \mu\), then \(\dim(C_j) =
  \mu_j + d_j = \mu + d_j\) so that \(\mu + d > \mu+d_j\) whence
  \(d_j < d\). Thus, we contradict the minimality of the pair \((\mu,
  d)\) by
  replacing \(B\) with the family \(C_1, \dots, C_q\) and by repeating
  the same operation for each \(B_i\) such that \(\mu_i = \mu\) and
  \(d_i = d\). 
\end{proof}

\begin{remark}
  The proof can be seen as a procedure that takes a representation of
  the form \(A = \Pi_n(B_1) \cup \dots \cup \Pi_n(B_p)\) and that
  produces a new one such that the pair \((\mu, d)\) decreases. Since
  this procedure can be applied as long as \(\mu > 0\), we can apply
  it repeatedly until \(\mu = 0\). Also, the set of pairs \((\mu, d)\)
  is well ordered so that the process must stop eventually. 
\end{remark}

The theorem along with Lemma \ref{global-parametrization} show that
every \(\Lambda\)-set has the \(\Lambda\)-Gabrielov property
(see \cite[p. 4380]{vdDries-Speissegger-generalized-power-series}). Thus,
if \(A, B \subset \real^n\) are two sub-\(\Lambda\)-sets, the set \(A
\setminus B\) is also a sub-\(\Lambda\)-set by Gabrielov's Theorem of
the complement \cite[Theorem
2.7]{vdDries-Speissegger-generalized-power-series}. Now, consider
\(\tau \colon \real \to (-1, 1)\) the definable map given by
\[\tau(x) = \frac{x}{\sqrt{1 + x^2}}.\]

\begin{thm}
  \label{thm:omin}
  A set \(A \subset \real^n\) is \(\struct\)-definable if and only if
  \(\tau^n(A) \subset (-1, 1)^n\) is a
  sub-\(\Lambda\)-set. Furthermore, the structure \(\struct\) is
  o-minimal and model complete. 
\end{thm}

\begin{proof}
  Consider \(\mathcal{D}_n\) to be the collection of subsets
  \(A \subset \real^n\) such that \((\tau^n)(A)\) is a
  sub-\(\Lambda\)-set. This collection is closed under
  unions and intersections. Furthermore, if \(A \subset (-1, 1)^n\) is a
  sub-\(\Lambda\)-set, so is \((-1, 1)^n \setminus A\) so that
  \(\mathcal{D}_n\) is also closed under complement. Finally, if \(A \in
  \mathcal{D}_{n+k}\), it is clear that \(\Pi_n(A) \in
  \mathcal{D}_n\). Thus, the collection \(\mathcal{D} =
  (\mathcal{D}_n)_{n \geq 0}\) is the collection of definable sets of a
  certain structure on \(\real\) which we write
  \(\real_\mathcal{D}\).

  We now need to argue that \(\mathcal{D}\) is exactly the
  collection of definable sets of the structure \(\struct\). Firstly,
  every sub-\(\Lambda\)-set is clearly definable in \(\struct\) so that
  every set \(A \in \mathcal{D}_n\) must also be definable in
  \(\struct\). Since \(\tau\) is order-preserving, the set
  \(\tau \times \tau(\graph <)\) is the graph of the order on \((-1, 1)^n\)
  which is clearly a sub-\(\Lambda\)-set so that \(\graph < \in
  \mathcal{D}_2\). Also, given \(x, y, z \in (-1, 1)\), we have
  \begin{align*}
    \tau^{-1}(x) + \tau^{-1}(y) = \tau^{-1}(z) \iff \exists &u, v, w >
    0, \\
    &(u^2, v^2, w^2) = (1-x^2, 1-y^2, 1-z^2) \\
    \land\  &xvw + yuw = zuv
  \end{align*}\
  Thus, \(\graph + \in \mathcal{D}_3\). Similarly, it is easy to show
  that \(\graph -, \graph \cdot \in \mathcal{D}_3\). Therefore, every
  polynomial is definable in the structure \(\real_\mathcal{D}\). In
  particular, the restriction of \(H\) to the complement of any
  neighborhood of \(0\) is
  \(\real_\mathcal{D}\)-definable. Thus, in order to prove that
  \(\real_\mathcal{D}\) and \(\struct\) have the same definable sets,
  it suffices to show that the restriction of \(H\) to some
  neighborhood of \(0\) is definable in \(\real_\mathcal{D}\).

  Up to considering \(H - H(0)\), we may as well assume that \(H(0)
  = 0\). In this case, the function \(\tau^{-1} \circ H \circ \tau\)
  is well defined in a neighborhood of \(0\). Furthermore,
  \(\tau \times \tau(\Gamma(\tau^{-1} \circ H \circ \tau))\) is the graph of \(H\)
  restricted to some neighborhood of \(0\). Since this is a
  \(\Lambda\)-set, it means that \(\tau^{-1} \circ H \circ \tau\)
  is \(\real_\mathcal{D}\)-definable. Finally, it is clear that
  \(\tau\) is a \(\real_\mathcal{D}\)-definable function so that
  \(\tau^{-1}\) is also \(\real_\mathcal{D}\)-definable. Putting
  everything together, we find that \(H\) is
  \(\real_\mathcal{D}\)-definable whence the result. 
  
  Since \(\tau\) is a homeomorphism and every sub-\(\Lambda\)-set has
  only finitely many connected components by Corollary
  \ref{cor:lambda-finite-components}, it follows that every
  \(\struct\)-definable set also has only finitely many connected
  components. Thus, \(\struct\) is o-minimal. Furthermore, consider
  \(A \subset \real^n\) a definable set. Then, there is \(A' \subset
  \real^{n+k}\) a \(\Lambda\)-set such that \(\Pi_n(A') =
  \tau^n(A)\). Then, \(A'\) is existentially definable by Proposition
  \ref{basic-simple-sub-lambda-set} so that the set
  \[B = \{(y, x) \in \real^{n+k} \times \real^{n+k} \for \tau^{n+k}(y)
    = x, x \in A\}\]
  is also existentially definable. Furthermore, \(\Pi_n(B) = A\) so
  that \(\struct\) is also model complete. 
\end{proof}
\appendix

\section{A counter example}
\label{appendix}

In this Appendix, we will follow the proof of \cite[Lemma
3.7]{legal-rolin-not-c-infty} in a special case. We will see that the
set \(A\) obtained in this way has empty germ at \(0\). This is an
issue as the proof of \cite[Proposition 3.8]{legal-rolin-not-c-infty}
freely restricts such sets \(A\) to arbitrary neighborhoods of \(0\). 

Let \(r = (1, 1)\) and 
\[M = \{(x, y) \in \polydisk 2 r \for y > x\}.\] 
Then, the projection \(\Pi_1\) on the \(x\) coordinate has constant
rank. As in the proof of Lemma 3.7, we define \(\varphi \colon I_r \to
\real\) as
\[\varphi(x, y) = (y-x)(1-x^2)(1-y^2)\]
and we let \(A = \{(x, y) \in M \for \varphi \restriction M_x \text{
  is critical at } y\}\). We can then compute
\[A = \left\{(x, y) \in M \for y = \frac{1}{3} x \pm
  \frac{\sqrt{x^2+3}}{3}\right\}.\]
There is \(\varepsilon > 0\) such that, for \(x \in (-\varepsilon,
\varepsilon)\), we have
\[ \abs{\frac{1}{3}x \pm \frac{\sqrt{x^2 +3}}{3}} > \frac{\sqrt
    2}{3}\]
Thus, \(A \cap \polydisk 2 {r'} = \varnothing\) where \(r' =
\left(\varepsilon, \frac{\sqrt 2}{3}\right)\).

\end{document}